\documentclass[a4paper]{article}

\usepackage{paralist}
\usepackage{bm}

\newcommand{\union}[2]{\mbox{$#1\!\cup\!#2$}}

\newtheorem{observation}{Observation}

\usepackage {color}
\usepackage{color,graphicx}
\usepackage{amsmath}
\usepackage{amssymb}
\usepackage{url}
\usepackage{hyperref}
\usepackage{authblk}
\usepackage[a4paper,hmargin=1.3in,vmargin=1.2in]{geometry}

\newtheorem{theorem}{Theorem}[section]
\newtheorem{proposition}[theorem]{Proposition}

\newtheorem{lemma}[theorem]{Lemma}
\newtheorem{corollary}[theorem]{Corollary}

\newcommand{\qed}{\hspace*{\fill} $\square$  \ifmmode \else
    \par\addvspace\topsep\fi}
\newenvironment {proof}{\par\addvspace\topsep\noindent{\it Proof.}
    \ignorespaces }{\qed}
\newcommand{\claimqed}{\hspace*{\fill} $\triangle$  \ifmmode \else
    \par\addvspace\topsep\fi}

\begin{document}

\title{Flips in Combinatorial Pointed Pseudo-Triangulations\\
with Face Degree at most Four}

\author[1]{Oswin Aichholzer}
\author[1]{Thomas Hackl}
\author[2]{David Orden}
\author[1]{Alexander Pilz}
\author[3]{Maria Saumell}
\author[1]{Birgit~Vogtenhuber}

\affil[1]{
Institute for Software Technology, Graz University of Technology.
Inffeldgasse 16b/II, Graz, Austria\\
\texttt{[oaich|thackl|apilz|bvogt]@ist.tugraz.at}
}
\affil[2]{
Departamento de F\'{\i}sica y Matem\'aticas, Universidad de Alcal\'a.
Apdo. de Correos 20, E-28871. Alcal\'a de Henares, Spain\\
\texttt{david.orden@uah.es}
}
\affil[3]{
Department of Mathematics and European Centre of Excellence NTIS (New Technologies for the Information Society), University of West Bohemia,\\
Univerzitn\'{\i} 22, 306 14 Plze\v{n}, Czech Republic\\
\texttt{saumell@kma.zcu.cz}
}

\date{}
\maketitle

\begin{abstract}
In this paper we consider the flip operation for combinatorial pointed pseudo-triangulations where faces
have size~3 or~4, so-called \emph{combinatorial 4-PPTs}.
We show that every combinatorial 4-PPT is stretchable to a geometric pseudo-triangulation, which in general is not the case if faces may have size larger than 4.
Moreover, we prove that the flip graph of combinatorial 4-PPTs is connected and has dia\/meter $O(n^2)$, even in the case of labeled vertices with fixed outer face.
For this case we provide an $\Omega(n\log n)$ lower bound.
\end{abstract}

\section{Introduction}
\label{sec:Introduction}

Given a graph of a certain class, a \emph{flip} is the operation of
removing one edge and inserting a different one such that the
resulting graph is again of the same class.
An example of such a class is the class of maximal
planar (simple) graphs, also called \emph{combinatorial triangulations}, where any combinatorial embedding (clockwise order
of edges around each vertex) has faces only of size~$3$.
%For the class of maximal
%planar (simple) graphs, any combinatorial embedding (clockwise order
%of edges around each vertex) has only faces of size~$3$ and hence is
%called a \emph{combinatorial triangulation}.
Flips in combinatorial
triangulations remove the common edge of two triangular faces and
replace it by the edge between the two vertices not shared by the
faces, provided that these two vertices were not already joined by an
edge.  Combinatorial triangulations have a geometric counterpart in
triangulations of point sets in the plane, which are maximal plane
geometric (straight-line) graphs with predefined vertex positions.  In
this geometric setting there is also a flip operation, for which a
different restriction applies: An edge can be flipped if and only if
the two adjacent triangles form a convex quadrilateral (otherwise the
new edge would create a crossing).

Flips in (combinatorial) triangulations have been thoroughly studied.
% See~\cite{flip_survey} for a survey.
See the survey by Bose and Hurtado~\cite{flip_survey}.
A prominent question about flips is to study the \emph{flip graph}.
This is an abstract graph whose vertices are the members of a given graph class having the same number of vertices,  and in which two graphs are neighbors if and only if one can be transformed  into the other by a single flip.
For both triangulations and combinatorial triangulations %(with fixed vertex positions)
the flip graph is connected.
Lawson~\cite{lawson_connected} showed that the flip graph of triangulations of a point set is connected with quadratic diameter, which was later shown to be tight~\cite{hurtado_noy_urrutia}.
For combinatorial triangulations there are actually two classes to consider: those of labeled and unlabeled graphs, where in the latter class no two distinct elements are isomorphic.
For unlabeled combinatorial triangulations on $n$ vertices Wagner~\cite{wagner} proved connectedness of the flip graph, and Komuro~\cite{komuro} showed its diameter to be~$\Theta(n)$.
For the labeled setting Sleator, Tarjan, and Thurston~\cite{labeled_triang} showed the diameter to be $\Theta(n \log n)$.

Triangulations have a natural generalization in pseudo-triangulations.
They have become a popular structure in Computational Geometry within the last two decades, with applications in, e.g., rigidity theory and motion planning.
See the survey by Rote, Santos, and Streinu~\cite{flip_survey}.
A \emph{pseudo-triangle} is a simple polygon in the plane with exactly three convex vertices (i.e., vertices whose interior angle is smaller than $\pi$).
A \emph{pseudo-triangulation}~$\mathcal{T}$ of a finite point set~$S$ in the plane is a partition of the convex hull of~$S$ into pseudo-triangles such that the union of the vertices of the pseudo-triangles is exactly~$S$.
Triangulations are a particular type of pseudo-triangulations, actually the ones with the maximum number of edges.
Those with the minimum number of edges are the so-called \emph{pointed pseudo-triangulations}, in which every vertex is \emph{pointed}, i.e., incident to a reflex angle (an angle larger than~$\pi$)~\cite{streinu}.

Flips can also be defined for the class of pseudo-triangulations of point sets in the plane.
The flip graph for general pseudo-triangulations is known to be connected~\cite{novel_type}, as well as the subgraph induced by pointed pseudo-triangulations~\cite{pt_connected}.
The currently best known bound on the diameter is $O(n \log n)$ for both flip graphs~\cite{novel_type,bereg}, where here and for the rest of the paper $n$ denotes the number of vertices.

In a pseudo-triangulation, the pseudo-triangles can have linear size.
Hence, in contrast to triangulations, the flip operation can no longer
be computed in constant time.
This fact led to the consideration of pseudo-triangulations in which
the size of the pseudo-triangles is bounded by a constant.
Kettner et al.~\cite{pt_degree_bound} showed that every point set admits a pointed pseudo-triangulation with face degree at most four  (except, maybe, for the outer face).
We call such a pointed pseudo-triangulation a \emph{4-PPT}.

On the one hand, 4-PPTs behave nicely for problems which are hard for general
pseudo-triangulations. For instance, they are always properly 3-colorable,
while 3-colorability is NP-complete to decide for general
pseudo-triangulations~\cite{3color_pt}.
On the other hand, known properties of general pseudo-triangulations
remain open for 4-PPTs. For instance, it is not known whether the
flip graph of 4-PPTs is connected, even for the basic case of a
triangular convex hull.

The aim of this paper is to make a step towards answering this last question,
by considering the combinatorial counterpart of 4-PPTs.

A \emph{combinatorial pseudo-triangulation}~\cite{osss-cpts-07} is
a combinatorial embedding of a planar simple graph in the plane together with
an assignment of tags \emph{reflex/convex} to its angles such that
\begin{inparaenum}[(1)]
\item every interior face has exactly three angles tagged convex,
\item all the angles of the outer face are tagged reflex, and
\item no vertex is incident to more than one reflex angle.
\end{inparaenum}
(These tags of the angles are called ``labels'' by Orden et al.~\cite{osss-cpts-07}, we use a different term to prevent confusion with the classic labels of the vertices.)
%\maria{What is a topological embedding? In particular, what is the difference between a topological embedding and a combinatorial embedding?}

Note that the assignment of these tags fulfills the same properties
as actual reflex/convex angles in a (geometric) pseudo-triangulation.
This analogy with the geometric case goes on by calling \emph{pointed}
vertices in a combinatorial pseudo-triangulation
those which are indeed incident to one angle tagged reflex.
Then, \emph{combinatorial pointed pseudo-triangulations} are those
in which every vertex is pointed. %Furthermore, one can consider only
Combinatorial pointed pseudo-triangulations with face degree
at most four (except, maybe, for the outer face),
will be called \emph{combinatorial 4-PPTs}.

As it has been done for combinatorial triangulations, we consider flip graph connectivity of the labeled and unlabeled graph;
while we allow the outer face (predefined by the combinatorial embedding in the plane) to have an arbitrary number of vertices, we require these vertices to be the same in the source and the target graph.

\section{Properties}
\label{sec:PropertiesC4PPTs}

In this section, we prove some properties of combinatorial 4-PPTs and, in particular, we show that every combinatorial 4-PPT is stretchable to a geometric pseudo-triangulation.

\begin{lemma}
\label{lemma:corners-genLaman}
Let $T$ be a combinatorial 4-PPT and~$H$ be a subgraph of~$T$ with~$|V(H)|\geq 3$.
Then~$H$ has at least~$3$ vertices whose reflex angle is contained
in the outer face of~$H$ (called ``corners of first type'' by Orden et al.~\cite{osss-cpts-07}).
\end{lemma}
\begin{proof}
W.l.o.g., we may assume that $H$ consists of a single connected component.
Let $H'$ be the maximal subgraph of $T$ that has the same outer face as $H$.
Hence, if the claim holds for $H'$ it also holds for $H$, and we only need to consider inner faces of size~3 or~4.
For the subgraph $H'$, let us denote with $n$ the number of vertices, $e$ the number of edges, $t$ the number of inner faces of size~$3$, $q$ the number of inner faces of size~$4$, $b$ the number of boundary angles and $c$ the number of convex boundary angles in the outer face of~$H'$.
Note that $b \geq 3$ and that $b > n$ is possible.

Let us double-count the edges.
On the one hand, the number of angles equals twice the number of edges; since there are~$n$ reflex angles and $3t+3q+c$ convex angles, we get that
$
2e=3t+3q+c+n.
$
On the other hand, from Euler's formula we have
$
e=n+t+q-1.
$
Eliminating~$e$ from these two equations, we get that the number of reflex angles is
$
n=t+q+2+c.
$
Now we can express the number~$n$ of reflex angles as $b-c+q$, to get that
$
b-c=t+2+c,
$
which is at least~$3$ if $c>0$. Either in this case or if $c=0$, we get that $b-c\geq 3$, as desired.
\end{proof}

\begin{corollary}
\label{corol:number_triangles}
In any combinatorial 4-PPT of the interior of a simple cycle with
$b$ vertices, of which $c$ have the reflex angle inside the cycle,
the number~$t$ of triangular faces is given by $t = b - 2c - 2$.
\end{corollary}

A combinatorial pseudo-triangulation has the \emph{generalized Laman property} if every subset of $x$ non-poin\-ted vertices and $y$ pointed vertices, where
%$x\!+\!y\!\geq\!2$,
$x+y\geq 2$,
induces a subgraph with at most $3x + 2y - 3$ edges.
Both this property and the number of reflex angles from Lemma~\ref{lemma:corners-genLaman} are related to the stretchability of a combinatorial pseudo-triangulation into a geometric one.
A face of a combinatorial pseudo-triangulation is called \emph{degenerate} if it contains edges which appear twice on the boundary of this face.
See \figurename~\ref{fig:NonStretchable5CPPT}~(left).
Note that in our setting this is equivalent to the definition by Orden et al.~\cite{osss-cpts-07} where a face is non-degenerate if the edges incident to it form a simple closed cycle.

% \maria{I find the following definition, from~\cite{osss-cpts-07}, to be more understandable: A face is non-degenerate if the edges incident to it form a simple closed cycle.}

%\begin{prop}{\protect~\cite[Corollary~2]{osss-cpts-07}}
\begin{proposition}[Orden et al.~\cite{osss-cpts-07}, Corollary~2]
\label{prop:equiv-prop}
The following properties are equivalent for a combinatorial pseudo-triangulation $G$:
%\begin{inparaenum}[(1)]
\begin{enumerate}
 \item $G$ can be stretched to become a pseudo-triangulation.
 \item $G$ has the generalized Laman property.
 \item $G$ has no degenerate faces and every subgraph of $G$ with at least
 three vertices has at least three corners of first type.
%\end{inparaenum}
\end{enumerate}
\end{proposition}
%\end{prop}

Since, by definition, combinatorial 4-PPTs have no degenerate
faces, we can use Proposition~\ref{prop:equiv-prop} to conclude
the following.

\begin{theorem}
\label{theorem:stretched}
Every combinatorial 4-PPT can be stretched to become a 4-PPT with the given assignment of angles.
Furthermore, combinatorial 4-PPTs have the generalized Laman property.
\end{theorem}

Note that there exist non-stretchable combinatorial pointed pseudo-triangulations with faces of size at most~5.
See \figurename~\ref{fig:NonStretchable5CPPT}~(right).
There and in the forthcoming figures, circular arcs denote angles tagged as reflex.

\begin{figure}[htb]
%\vspace{-2ex}
\centering
\includegraphics{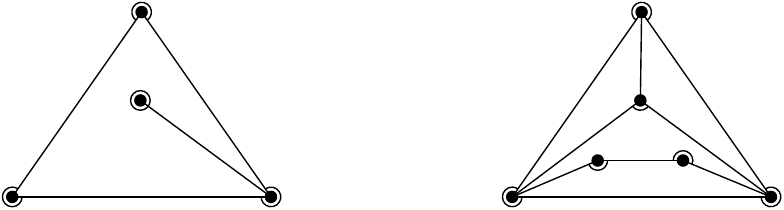}
\caption{Left: A degenerate 5-face. Right: A non-stretchable combinatorial pointed pseudo-triangulation\protect~\cite{osss-cpts-07}.}
\label{fig:NonStretchable5CPPT}
\vspace{-1ex}
\end{figure}

\section{Flips}
\label{sec:flips}

% In the following we focus on combinatorial 4-PPTs with a fixed triangular outer face.
% For such a combinatorial 4-PPT, Corollary~\ref{corol:number_triangles} implies that there is only one  interior triangular face.
Before defining flips between combinatorial 4-PPTs, we make some observations about their geometric counterpart.
For good visual distinction, we draw the edges of non-geometric graphs as non-straight Jordan arcs throughout this section.

\begin{figure}[htb]
\centering
\includegraphics{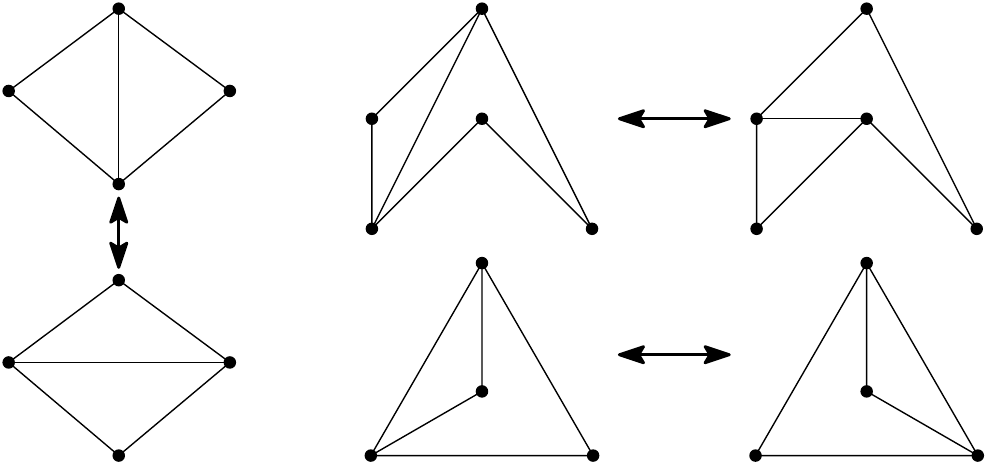}
\caption{Geometric flip of an edge of a triangle. Left: Both faces are triangles. Right: One face is a quadrilateral and in the lower case the removal of the flipped edge gives a degenerate 5-face.}
\label{fig:GeometricFlip}
\vspace{-2ex}
\end{figure}

% Geometric 4-PPTs with triangular convex hull also have only one interior triangle.
% Furthermore, every edge of the triangle (except for those being part of the convex hull) is flippable~~\cite{pt_survey}.
In geometric 4-PPTs every edge of a triangle (except for those being part of the convex hull) is flippable~\cite{pt_survey}.
Consider flipping an edge~$e$ which separates a triangle~$\triangle$ from another face~$F$ in a geometric 4-PPT.
If $F$ is also a triangle, then removing $e$ and inserting the other diagonal $e'$ of the convex 4-face $\triangle \cup F$ is the well known ``Lawson flip".
(Note that $\triangle \cup F$ has to be convex because of the pointedness of the 4-PPT.)
If $F$ is a 4-face, then the removal of $e$ merges $\triangle$ and $F$ into a 5-face, which might be degenerate if $\triangle$ and $F$ share two edges.
Note that this degenerate case is the only one in which $\triangle$ and $F$ can share three vertices, as there are no multiple edges in geometric graphs.
See \figurename~\ref{fig:GeometricFlip}.

Similar to the geometric case, we consider flips of an interior edge $e$ of an interior triangular face $\triangle$ in a combinatorial 4-PPT:
Consider the face~$F$, triangular or quadrangular, sharing~$e$ with~$\triangle$.
A \emph{flip} of~$e$ consists in replacing~$e$ by another edge~$e'$ such that
\begin{inparaenum}[(1)]
  \item $e'$ splits $\union{(\triangle}{F)}\setminus e$ into a triangular face~$\triangle'$ and a face~$F'$, triangular or quadrangular, respectively, and
  \item the result is a combinatorial 4-PPT.
\end{inparaenum}
In particular, and in contrast to the geometric case, we have to explicitly avoid multiple edges in the combinatorial setting.
Hence, we have to ensure that the edge $e'$ that is inserted by the flip is not already contained in the combinatorial 4-PPT (as an edge outside $\union{\triangle}{F}$).
To emphasize that an exchange of two edges is a flip avoiding multiple edges, we sometimes
call a flip \emph{valid}.
Further, to highlight that an exchange of two edges which would locally (inside $\union{\triangle}{F}$) be a flip would introduce multiple edges, we call this an \emph{invalid} flip.
Recall though, that a flip is defined to be valid and we use this distinction only for emphasis in situations where we prove the existence of certain flips.

Observe that, in a (combinatorial) 4-PPT, if one face involved in a flip is triangular, then after the flip no face can have more than four vertices.
Thus, we restrict ourselves to flips where at least one involved face is triangular.
The following lemma shows that every interior edge of an interior triangular face can be flipped.

\begin{lemma}
\label{lemma:existanceofflips}
In a combinatorial 4-PPT, every edge~$e$ of an interior triangular face that is not an edge of the outer face is flippable.
Furthermore:
\begin{inparaenum}[(1)]
  \item If the removal of~$e$ results in a 4-face or a degenerate 5-face, then there is a unique valid flip for~$e$.
  \item If removing~$e$ results in a non-degenerate 5-face, then there are at least two valid flips for~$e$.
\end{inparaenum}
\end{lemma}

% \maria{The first claim of the lemma follows from the fact that every combinatorial 4-PPT is stretchable to a geometric pseudo-triangulation and the fact that in geometric 4-PPTs every edge of a triangle (except for those being part of the convex hull) is flippable; is that right? If so, should we mention it?}

\begin{proof}
Let $\triangle$ be a triangular face and let $F$ be the face that is separated from $\triangle$ by $e$.
If $F$ is also a triangular face, then $\triangle \cup F$ is a 4-face.
In this case exchanging~$e$ by the unique other diagonal of $\triangle \cup F$ is a valid flip.
See \figurename~\ref{fig:CombinatorialFlipsDegenerated5-face}~(left).
If $F$ is a face of size 4 we have to distinguish two cases.
The first case is when $\triangle \cup F$ is degenerate.
Then there is only one choice of~$e'$ in order to split $\triangle \cup F \setminus e$ as required.
Furthermore, the corresponding edge~$e'$ could not already be an edge, since it was not in the interior of~$\triangle \cup F$ and it cannot go through the exterior of~$\triangle \cup F$ because of planarity.
See \figurename~\ref{fig:CombinatorialFlipsDegenerated5-face}~(right).
Hence, this choice is always valid.

\begin{figure}[htb]
\begin{center}
\includegraphics{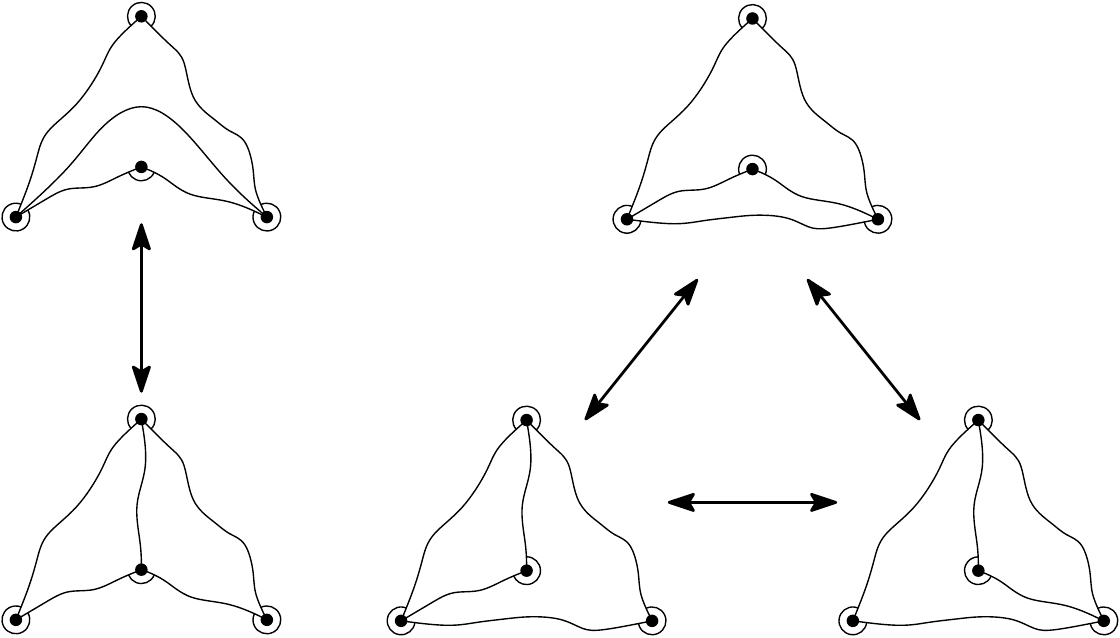}
\caption{Combinatorial flips for an edge of a triangular face. Left: Both faces are triangular. Right: One face is a quadrangular and the 5-face is degenerate.}
\label{fig:CombinatorialFlipsDegenerated5-face}
\end{center}
\end{figure}

The second case is when $\triangle \cup F$ is non-degenerate.
We show that flipping towards an edge using the reflex vertex is always valid.
See \figurename~\ref{fig:CombinatorialFlips}.
Denote by $v_1,\ldots,v_5$ the boundary vertices of~$\triangle \cup F \setminus e$, in counterclockwise order and with $v_1$ being the reflex vertex.
The edge~$e'$ we intend to insert is then either $v_1v_3$ or $v_1v_4$.
Let us focus on the first case, the other one being handled analogously.
If $e'=v_1v_3$ is not valid, there already has to be an edge between~$v_1$ and~$v_3$ in the exterior of~$\triangle \cup F$.
But then at most two vertices of the 3-cycle $v_1 v_2 v_3$ have their reflex angle on the outside of that cycle, by this contradicting Lemma~\ref{lemma:corners-genLaman}.
See \figurename~\ref{fig:GoodFlips}.

\begin{figure}[htb]
\begin{center}
\includegraphics{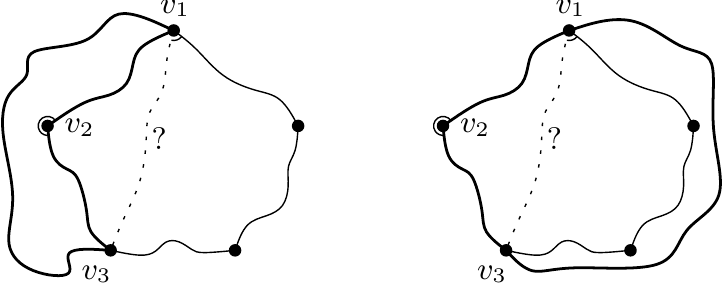}
\caption{Flipping towards an edge incident to the reflex vertex is always valid.}
\label{fig:GoodFlips}
\end{center}
\end{figure}

It remains to prove that in the non-degenerate case there are at least two valid flips for~$e$.
\figurename~\ref{fig:CombinatorialFlips} shows the possible flips when $\triangle \cup F$ is non-degenerate, with solid arrows indicating always valid flips and dotted arrows indicating flips which might be valid or not.

\begin{figure}[htb]
\begin{center}
\includegraphics{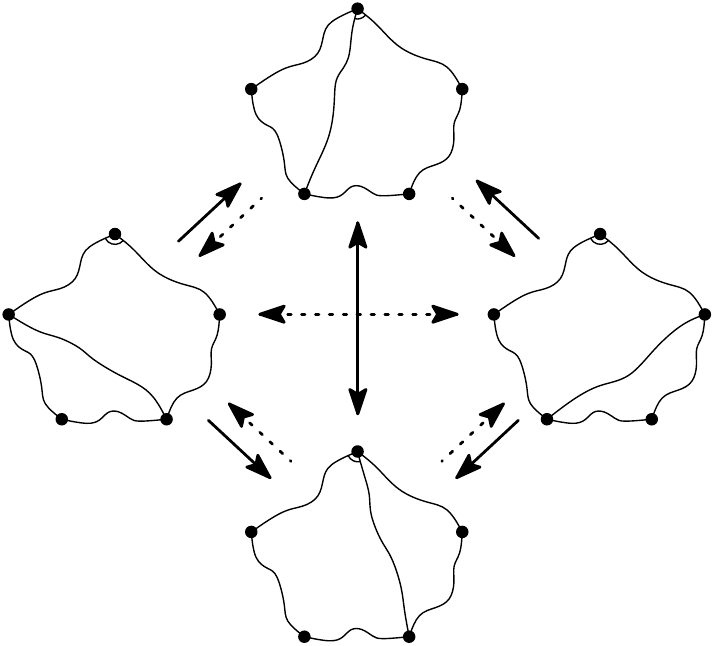}
\caption{Combinatorial flips for an edge of a triangular face in a combinatorial 4-PPT, non-degenerate case.}
\label{fig:CombinatorialFlips}
\end{center}
\end{figure}

If~$e$ is not incident to the reflex vertex, then there are two valid flips towards edges incident to that vertex.
If~$e$ is incident to the reflex vertex, there is always a valid flip towards the other diagonal~$e'$ incident to that vertex.
For a second valid flip, we show that the two remaining diagonals cannot simultaneously give invalid flips.
Let the edge to flip be $e=v_1v_3$ (the other case is analogous).

\begin{figure}[htb]
\begin{center}
\includegraphics{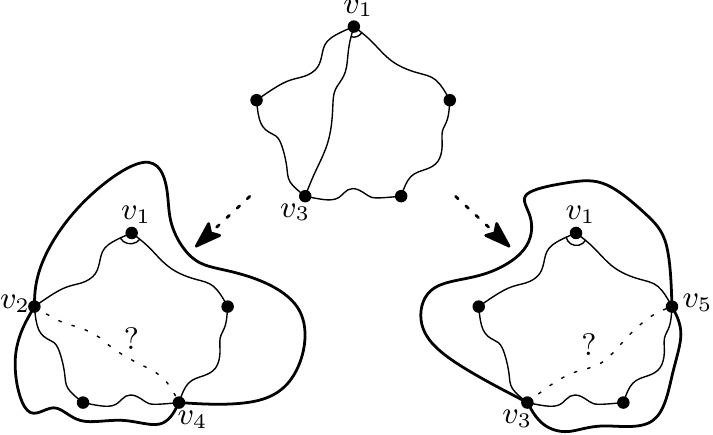}
\caption{In a non-degenerate 5-face, the two diagonals not using the reflex vertex cannot both give invalid flips.}
\label{fig:NotTwoBad}
\end{center}
\end{figure}

In order for both $v_2v_4$ and $v_3v_5$ to give invalid flips, the combinatorial 4-PPT must have both edges $v_2v_4$ and $v_3v_5$ in the exterior of the 5-face $v_1,\ldots,v_5$.
This is impossible since it would imply a crossing.
See \figurename~\ref{fig:NotTwoBad}.
\end{proof}

\begin{figure}[htb]
\begin{center}
\includegraphics{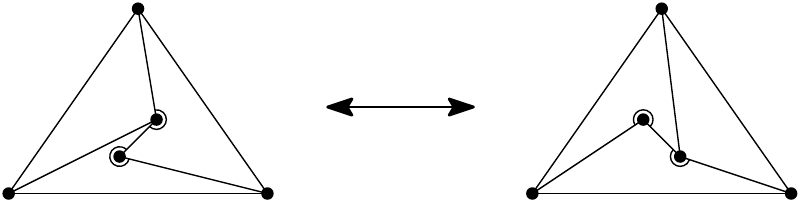}
\caption{Combinatorial flips might need different geometric embeddings.}
\label{fig:ObsGeomCombFlip}
\end{center}
\end{figure}

Observe that by Theorem~\ref{theorem:stretched}, every combinatorial 4-PPT can be stretched into a geometric 4-PPT.
Thus, the statement in Lemma~\ref{lemma:existanceofflips} that every interior edge of a triangular face is flippable can also be seen from the geometric case.
In contrast to the geometric case, where all valid flips are unique, the combinatorial case described in part~(2) of the lemma has up to three possible flips, of which there are always at least two valid ones.
This hints at another interesting observation.
Given a combinatorial flip between two combinatorial 4-PPTs, by Theorem~\ref{theorem:stretched} we know that both of them can be stretched into geometric 4-PPTs with straight edges.
However, it might not be possible to use the same geometric embedding for the vertices in both of them.
See \figurename~\ref{fig:ObsGeomCombFlip} for an example where two different geometric embeddings are required.

\section{Connectivity of the Flip Graph}
\label{sec:GraphOfFlips}

A usual approach for proving connectivity of the flip graph and bounding its diameter is to define a special canonical graph and to show that there exists a sequence of a certain (bounded) number of flips from any graph to the canonical one.
By the reversibility of the flips used in the sequence, this proves connectivity of the flip graph and gives a bound on its diameter.
In this section we deal with combinatorial 4-PPTs on unlabeled vertices with a fixed triangular outer face.
In later sections we will extend this base case to labeled vertices (Section~\ref{sec:labeledconnectedness}) and to the general case of combinatorial 4-PPTs with an arbitrarily sized outer face on unlabeled and labeled vertices (Section~\ref{sec:generalconnected}).

We define the unique \emph{canonical} combinatorial 4-PPT with triangular outer face to be the combinatorial 4-PPT where two of the vertices in the outer face are adjacent to all other vertices, while the third one has degree 2.
An example can be found in \figurename~\ref{fig:CanonicalSpinal}~(left).
Observe that this canonical combinatorial 4-PPT is indeed unique as we consider unlabeled vertices for now.
In the following we will show step by step how to build the flip sequence from any combinatorial 4-PPT to the canonical one.
We only allow flips of interior edges of interior triangular faces, as by Lemma~\ref{lemma:existanceofflips} these are always flippable.
For a combinatorial 4-PPT with triangular outer face, Corollary~\ref{corol:number_triangles} implies that there is only one interior triangular face.
Hence, with the presented flip sequences we will ``move" the single triangular face through the combinatorial 4-PPT.

Note that the following lemma is not restricted to a triangular outer face as the proof is (almost) the same for arbitrary sized outer faces, which will be needed also in Section~\ref{sec:generalconnected}.

\begin{lemma}
\label{lemma:triangletochgeneral}
Let~$T$ be a combinatorial 4-PPT with outer face $o_1,\ldots,o_h$, $h\geq3$.
For any edge~$b$ of the outer face, there is a sequence of $O(n)$ flips resulting in a combinatorial 4-PPT with one interior triangular face incident to~$b$.
\end{lemma}
\begin{proof}
Consider the dual of $T$ and choose a path $\triangle = F_0\to F_1\to F_2\to\cdots\to F_k=F_{\text{outer}}$ from an interior triangular face $\triangle$ to the outer face $F_{\text{outer}}$, such that $F_{k-1}$ and $F_{k}$ share the edge~$b$.
Let $e_j$, $1\leq j\leq k$, be an edge separating the faces $F_{j-1}$ and $F_{j}$ in the path (note that there might be two edges shared by two faces).
If there is a triangular face $F_{j}$ ($1\leq j\leq k-1$) in the path, then choose the triangular face $F_{j'}$ with highest index $j'$ among all triangular faces $F_{j}$, and replace $\triangle$ and the path in the dual of $T$ by $F_{j'}$ and the shortest (sub)path starting from this new~$\triangle$.
Note that this can only happen if $h>3$, as with a triangular outer face there exists only one triangle.

\begin{figure}[htb]
\centering
\includegraphics{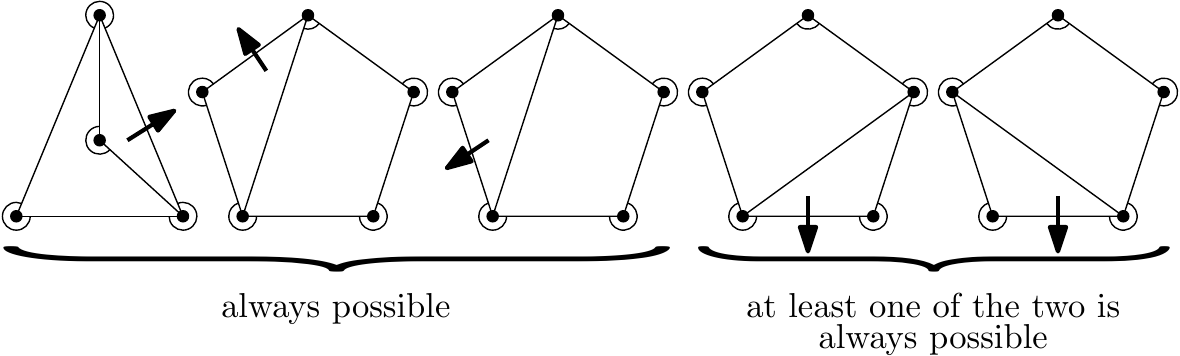}
\caption{The region $\triangle\cup F_j$ after the flip. The arrows indicate the edge $e_{j+1}$ to which the triangular face should be incident after the flip.
  Leftmost: $\triangle\cup F_j$ is degenerate.
  All but leftmost: The four different cases (depending on the relative position of $e_{j+1}$ and the reflex interior angle) for non-degenerate $\triangle\cup F_j$.
  The two rightmost 4-PPTs show cases with two possible flips, of which at most one can lead to a double edge.}
\label{fig:generalpathflip}
\end{figure}

We define the sequence of flips in such a way that after the $j$-th flip, $\triangle$ is incident to~$e_{j+1}$, which then separates $\triangle$ from~$F_{j+1}$.
Thus, after $k-1=O(n)$ flips $\triangle$ will be incident to $F_k=F_{\text{outer}}$ through~$e_{k}=b$, as required.

At the $j$-th flip we consider the region $\triangle\cup F_j$ and we have to replace an edge~$e$ shared by~$\triangle$ and~$F_j$ with a valid edge~$e'$, such that the new triangular face~$\triangle'$ is incident to~$e_{j+1}$.
If $\triangle\cup F_j$ is degenerate, then there are two edges shared by~$\triangle$ and~$F_j$.
For each of these edges there exists a unique valid flip, by Lemma~\ref{lemma:existanceofflips}.
Flipping the edge which does not share a vertex with~$e_{j+1}$ yields the desired result.
See \figurename~\ref{fig:generalpathflip} (left).
If $\triangle\cup F_j$ is non-degenerate, then the single edge shared by~$\triangle$ and~$F_j$ is $e_{j}$.
By Lemma~\ref{lemma:existanceofflips} there exist two valid flips for $e_{j}$.
At least one of these flips introduces an edge~$e'$, such that the new triangular face~$\triangle'$ is incident to~$e_{j+1}$.
See \figurename~\ref{fig:generalpathflip} (all but leftmost).
\end{proof}

Once the interior triangular face is incident to an edge~$b$ of the outer face, the next step will be flipping away interior edges incident to one endpoint of~$b$.

\begin{lemma}
\label{lemma:cleartip}
Given a combinatorial 4-PPT with triangular outer face in which the interior triangular face~$\triangle$ is incident to the edge~$b$ of the outer face, there is a sequence of flips resulting in a combinatorial 4-PPT in which the endpoint named~$t$ of~$b=rt$ has no interior incident edges.
\end{lemma}
\begin{proof}
We describe a flip sequence that removes all inner edges incident to $t$.
This flip sequence can be partitioned into two phases and some cases.
Let the vertices neighbored to the vertex $t$ be ordered radially around~$t$, starting with $r$.
In each case, let the vertices in that order be $r = w_0, \dots, w_k$.

\noindent \textbf{Phase~1:}
During this phase, the inner triangular face $\triangle$ has $rt$ as a side; i.e., $\triangle=trw_1$.
We distinguish three different cases:

\noindent \textbf{Case~1: $\bm{t w_1}$ is the only inner edge incident to~$\bm{t}$; i.e., $\bm{k=2}$.} %$w_2 = w_k = t$.
If $\triangle$ is incident to only one 4-face $F$ (i.e., $\union{\triangle}{F}$ is degenerate), we can flip the edge $tw_1$ and we are done.
Otherwise, let the 4-face~$F$ incident to $t w_1$ be $t w_1 u w_2$.
See \figurename~\ref{fig:ClearTipPhase1Case1}. The reflex angle inside $F$ is either at~$u$ or~$w_1$.
If it is at~$u$, we flip $t w_1$ to $w_0 u$, obtaining the 4-face $t w_0 u w_2$.
Otherwise, the reflex angle is at~$w_1$ and we flip $t w_1$ to $w_1 w_2$, obtaining the 4-face $t w_0 w_1 w_2$.
Either way, the degree of $t$ is 2 and we are done.
\begin{figure}[htb]
\vspace{-0.5ex}
\centering
\includegraphics{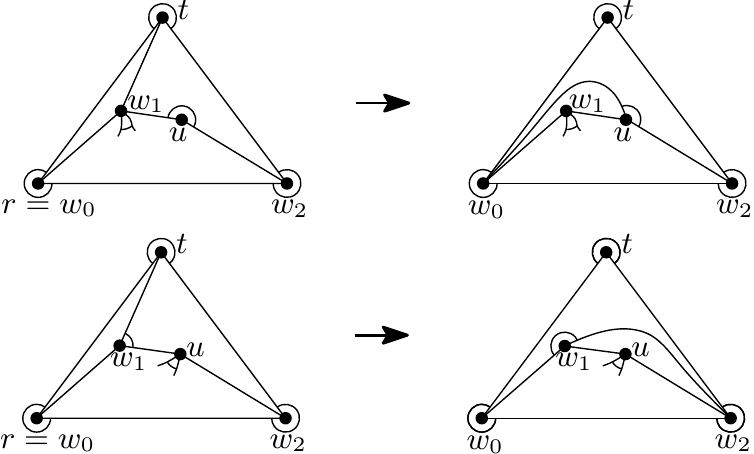}
\caption{Phase~1, Case~1: Only one interior edge is incident to~$t$.}
\label{fig:ClearTipPhase1Case1}
\vspace{-0.5ex}
\end{figure}

\noindent \textbf{Case~2: at least two inner edges are incident
to~$\bm{t}$ and there does not exist an edge $\bm{w_0 w_2}$.} See
\figurename~\ref{fig:ClearTipPhase1Case2}. Since the reflex angle
of $t$ is at the outer face we can replace the edge $t w_1$ by
$w_0 w_2$. This reduces the degree of $t$ by one. The inner
triangular face is again adjacent to $w_0t$, and we remain in
Phase~1.
\begin{figure}[htb]
\vspace{-0.5ex}
\centering
\includegraphics{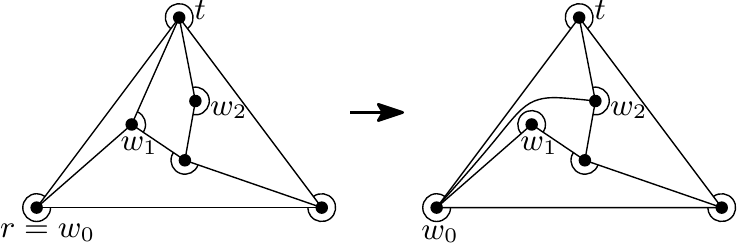}
\caption{Phase~1, Case~2: Several interior edges are incident to $t$ and ${w_0 w_2}$ does not exist.}
\label{fig:ClearTipPhase1Case2}
\vspace{-0.5ex}
\end{figure}

\noindent \textbf{Case~3: at least two inner edges are incident
to~$\bm{t}$ and there exists an edge $\bm{w_0 w_2}$.} See
\figurename~\ref{fig:ClearTipPhase1Case3}. If the two inner edges
of $\triangle$ are incident to a single 4-face, we have a
degenerate case; we flip the edge $w_0 w_1$ to $w_1 w_2$, making
$t w_1 w_2$ the inner triangular face. Otherwise, let the 4-face
$F$ incident to $t w_1$ be $t w_1 u w_2$; we flip $t w_1$ to $t u$
(this is possible since if $t u$ already existed, it would have to
cross the cycle $r w_2 u w_1$). Either way, the flip does not
reduce the degree of $t$, but the inner triangular face is now
inside the 3-cycle $t w_0 w_2$. We switch to Phase~2.
\begin{figure}[htb]
\centering
\includegraphics{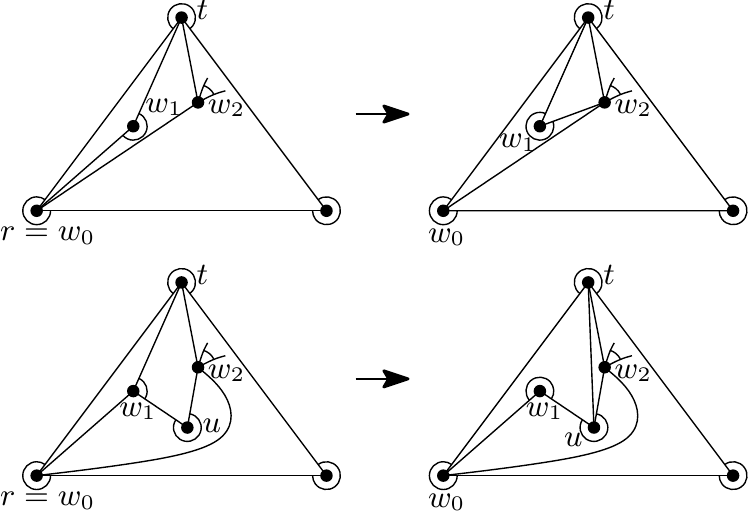}
\caption{Phase~1, Case~3: The possible transitions to Phase~2.}
\label{fig:ClearTipPhase1Case3}
%\vspace{-1ex}
\end{figure}

\noindent \textbf{Phase~2:}
During this phase, the inner triangular face is ${t w_1 w_2}$, and $w_1$ stays fixed for the whole phase.
Further, we know that $w_1$ was enclosed by a 3-cycle (at the transition to this phase), which implies that there are no edges from $w_1$ to $w_j$ for any $j \geq 2$.
We decrease the degree of $t$ in the following manner.

\noindent \textbf{Case~1: there is a 4-face $\bm{F}$ incident to $\bm{t w_2}$.}
There cannot be an edge $w_1 w_3$ since $w_1$ was enclosed by a 3-cycle.
Further, the reflex angle of $F$ is not at~$t$.
Hence, we can flip $t w_2$ to $w_1 w_3$, which reduces the degree of $t$ and we remain in Phase~2, with $t w_1 w_3$ being the new inner triangular face.

\noindent \textbf{Case~2: there is no 4-face incident to $\bm{t w_2}$; i.e., $\bm{k=2}$.} %$w_2 = w_k = t$.
This case is symmetric to Case~1 of Phase~1.
The edge $t w_1$ is flipped in one of the two ways described, reducing the degree of~$t$ to~2 and thus ending the process.
\end{proof}

\begin{figure}[htb]
\centering
\includegraphics{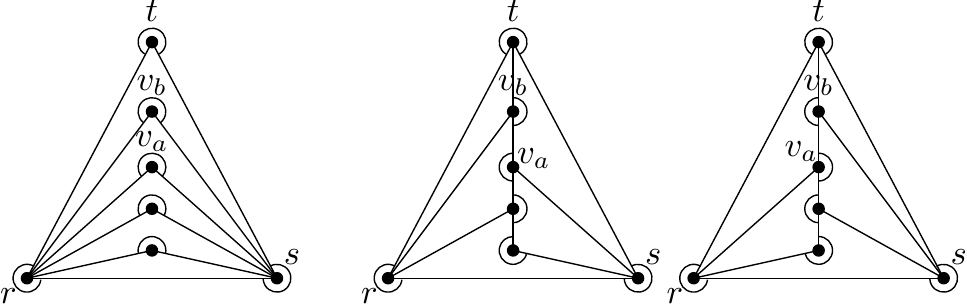}
\caption{Left: The canonical combinatorial 4-PPT $T$, with outer face $rst$ and base edge $rs$.
Middle and right: The two spinal combinatorial 4-PPTs (from $T$) with tip $t$, $s$-spinal (middle) and $r$-spinal (right).}
\label{fig:CanonicalSpinal}
\end{figure}

\begin{theorem}
\label{thm:connectedness}
The graph of flips in combinatorial 4-PPTs with $n$ vertices and triangular outer face is connected and has diameter $O(n^2)$.
\end{theorem}
\begin{proof}
Given such a combinatorial 4-PPT, follow the steps in Lemmas~\ref{lemma:triangletochgeneral} and~\ref{lemma:cleartip}, then use induction for the combinatorial 4-PPT obtained by removing~$t$.
This leads to the unique \emph{canonical} combinatorial 4-PPT with triangular outer face, where two of the vertices in the outer face are adjacent to all other vertices, while the third one has degree 2.
See \figurename~\ref{fig:CanonicalSpinal}~(left).

% \begin{figure}[htb]
% \centering
% \includegraphics{./figures/Canonical}
% \caption{A canonical combinatorial 4-PPT.}
% \label{fig:Canonical}
% \vspace{-1ex}
% \end{figure}

Furthermore, the number of flips needed in Lemmas~\ref{lemma:triangletochgeneral} and~\ref{lemma:cleartip} is
at most linear in the number of vertices of the combinatorial 4-PPT.
\end{proof}

% The presented basic case of combinatorial 4-PPTs with triangular outer face is extendible to an arbitrarily sized outer face, to labeled vertices, and also to the general case of combinatorial 4-PPTs with an arbitrarily sized outer face on labeled vertices.
% Elaborating on these extensions would go beyond the scope of this extended abstract, though.
% Details (and omitted proofs) can be found in a forthcoming full version. %of this work.

% new by -TH-: 20130826

\section{Connectivity for Labeled Combinatorial 4-PPTs}
\label{sec:labeledconnectedness}

The canonical graph produced in the proof of Theorem~\ref{thm:connectedness} does not care about the order of the interior vertices with respect to the extremal vertices.
When labeling the vertices of both the source and target graph accordingly, the two graphs produced are isomorphic and have the same combinatorial embedding, but might not be equivalent when arbitrary predefined labels are considered.
In this section we describe how to flip between canonical combinatorial 4-PPTs with a fixed triangular outer face under consideration of the labels.
In contrast to the classic setting of combinatorial triangulations, note that we insist on the outer face to be fixed, i.e., the vertices of the outer face occur in the same order (with the same order of labels) in all labeled combinatorial 4-PPTs of the flip sequence.

The canonical combinatorial 4-PPT, as exemplified in \figurename~\ref{fig:CanonicalSpinal}~(left), induces a total order on the interior vertices (i.e., vertices that are not incident to the outer face) by inclusion of vertices in 3-cycles formed by one interior vertex and the two vertices $r$ and $s$ of high degree.
Thus, given a canonical combinatorial 4-PPT, we say that an interior vertex $v_b$ is \emph{above} another interior vertex $v_a$ (and $v_a$ is \emph{below} $v_b$) if and only if the 3-cycle defined by $v_b$ (and $r$ and $s$) contains $v_a$ in its interior.
Further, two vertices $v_a$ and $v_b$ are \emph{neighbored} when they are neighbored in the total order.
% if and only if $v_a$ is above $m_a$ many points, $v_b$ is above $m_b$ many points, and $|m_a-m_b|=1$.
%
Let $v_1,\ldots,v_i$ be the $i=n-3$ interior vertices in that order.

Besides the canonical form, a second special class of combinatorial 4-PPTs that we will use is the one of \emph{spinal} combinatorial 4-PPTs; see \figurename~\ref{fig:CanonicalSpinal}~(middle and right).
In a spinal combinatorial 4-PPT the subgraph on $\{v_1,\ldots,v_i\}\cup\{t\}$ is a path with $t$ and $v_1$ as the end vertices.
This path is called the \emph{spine}.
Further, $\{v_1,\ldots,v_i\}$ are alternatingly (in the order on the spine) connected to $r$ and $s$ to complete a combinatorial 4-PPT.
The reflex angle at $v_k$, $2\leq k\leq i$ is the angle between the two edges of the spine (incident to $v_k$).
The reflex angle at $v_1$ is inside the face defined by $r, s, v_1$, and $v_2$ (if $i=1$, then $v_2=t$).
Depending on whether $v_1$ is connected to $r$ or $s$ we distinguish between an $r$-spinal or $s$-spinal combinatorial 4-PPT, respectively.

\begin{figure}[htb]
\centering
\includegraphics{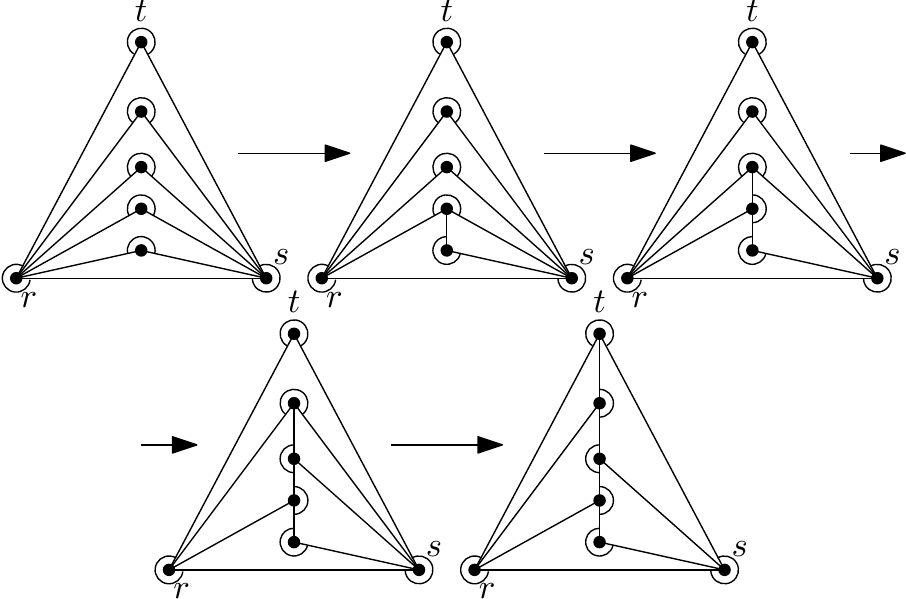}
\caption{Flipping from the canonical combinatorial 4-PPT to the $s$-spinal combinatorial 4-PPT.}
\label{fig:BetweenCanonicalSpinal}
\end{figure}

Observe that there exists a simple sequence of $i$~flips to transform a canonical combinatorial 4-PPT to the $r$- or $s$-spinal combinatorial 4-PPT.
See \figurename~\ref{fig:BetweenCanonicalSpinal} for an example of flipping to the $s$-spinal combinatorial 4-PPT.
Flipping to the $r$-spinal combinatorial 4-PPT is analogous, but flipping the edge $sv_1$ in the first step.
It is easy to see that the total order in the canonical combinatorial 4-PPT is the same as the one on the spine, for the two spinal combinatorial 4-PPTs:
Two vertices are neighbored on the spine if and only if they are neighbored in the total order of the canonical combinatorial 4-PPT.

\begin{observation}
\label{obs:CanonicalSpinal}
For a triangular outer face $rst$ and $i$ interior vertices, flipping from a canonical combinatorial 4-PPT with base edge $rs$ to an $r$- or $s$-spinal combinatorial 4-PPT can be done in $i$ flips.
The order in the canonical combinatorial 4-PPT equals the order on the spine. %and vice versa \maria{Why vice versa?}.
\end{observation}

Let $T$ be some canonical labeled combinatorial 4-PPT with triangular outer face $rst$, base edge $rs$, and $i$ interior vertices $\{v_1,\ldots,v_i\}$.
If $i<2$ then reordering of the interior vertices is not necessary.
For reordering the $i\geq2$ interior vertices in our labeled setting we need to be able to exchange two neighbored labeled vertices (for which we will use the spine).
We call the required sequence of flips a \emph{swap}.
For swapping two labeled vertices $v_k$ and $v_{k+1}$ in a spinal combinatorial 4-PPT (and thus also in a canonical combinatorial 4-PPT) we need to distinguish the three cases $k=1$, $k=2$, and $3\leq k< i$.
% to distinguish 3 cases:
% \begin{enumerate}
%  \item swap $k=1$,
%  \item $k=2$, and
%  \item $3\leq k\leq i-1$.
% \end{enumerate}
%
If $k \leq i-2$, let $t' = v_{k+2}$; otherwise, let $t' = t$.
For all cases we consider the subset $\{v_1,\ldots,v_{k+1}\}\cup\{r\}\cup\{s\}\cup\{t'\}$.
%, where $t'$ is either $v_{k+2}$ if $k\leq i-2$ or $t$, otherwise.
%
%
Further, we assume that we have already flipped to the spinal combinatorial 4-PPT with outer face $rst'$.
% For brevity, we also write the combinatorial 4-PPT is spinal in $rst'$.
Note that the subgraph in $rst'$ is a spinal combinatorial 4-PPT.
(By Observation~\ref{obs:CanonicalSpinal}, flipping from a canonical combinatorial 4-PPT to this situation takes $k+1$ flips.)

\begin{figure}[htb]
\centering
\includegraphics{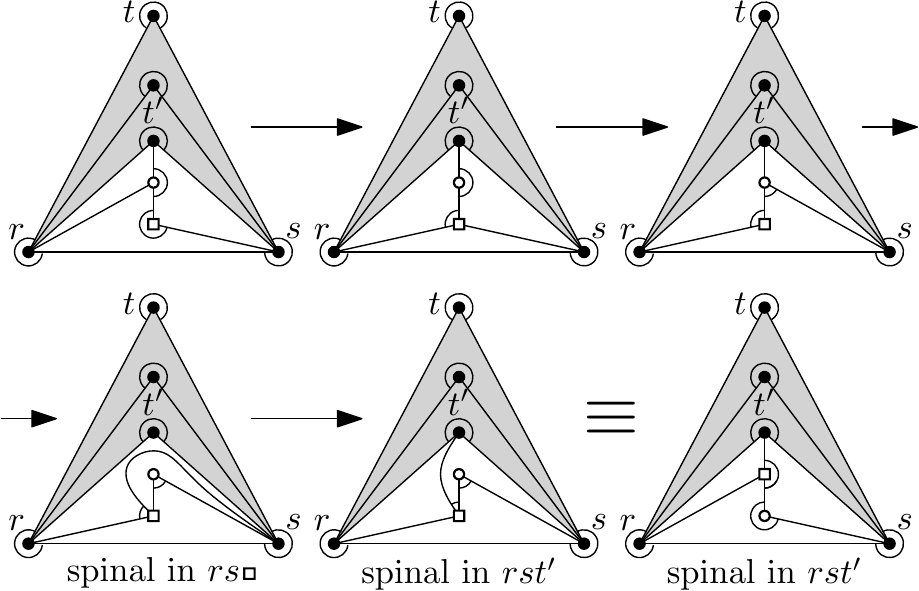}
\caption{Swapping the position of $v_1$ and $v_2$ in a combinatorial 4-PPT that is spinal in $rst'$.}
\label{fig:firstswap}
\end{figure}

The flip sequence for the case $k=1$ is depicted in \figurename~\ref{fig:firstswap}.
The two vertices $v_1$ and $v_2$ are shown as a white square and a white dot, respectively, which depict the different labels of the vertices.
After three flips we can reach a combinatorial 4-PPT that is spinal in $rsv_2$, where the labeled vertices $v_1$ and $v_2$ have been swapped (\figurename~\ref{fig:firstswap}, second row, left).
Recall that the numbering of the vertices $v_k$, $1\leq k\leq i$, is defined by their position along the spine; e.g., $v_1$ is always the first vertex on the spine.
Observe that this combinatorial 4-PPT is only one additional flip away from the canonical combinatorial 4-PPT (with swapped labels).
Applying an additional flip, we can reach a combinatorial 4-PPT that is spinal in $rst'$, where the labeled vertices $v_1$ and $v_2$ have been swapped (\figurename~\ref{fig:firstswap}, second row, middle and right).
%
%
% After three flips we can reach a combinatorial 4-PPT that is spinal in $rst'$, where the labels of $v_1$ and $v_2$ have been swapped (\figurename~\ref{fig:firstswap}, second row, left and middle).
% Recall that the numbering of the vertices $v_k$, $1\leq k\leq i$, is defined by their position along the spine; e.g., $v_1$ is always the first vertex on the spine.
% Applying a different third flip to the flip sequence, we can reach a combinatorial 4-PPT that is spinal in $rsv_2$, where the labels of $v_1$ and $v_2$ have been swapped (\figurename~\ref{fig:firstswap}, second row, right).
% This combinatorial 4-PPT is only one additional flip away from the canonical combinatorial 4-PPT (with swapped labels), shown in \figurename~\ref{fig:firstswap}, third row.

The flip sequences for the cases $k=2$ and $3\leq k< i$ are very similar.
In fact, the case $k=2$ is just a special case of the case $3\leq k< i$.
For completeness, the flip sequence for the case $k=2$ is given in \figurename~\ref{fig:secondswap}.

\begin{figure}[htb]
\centering
\includegraphics{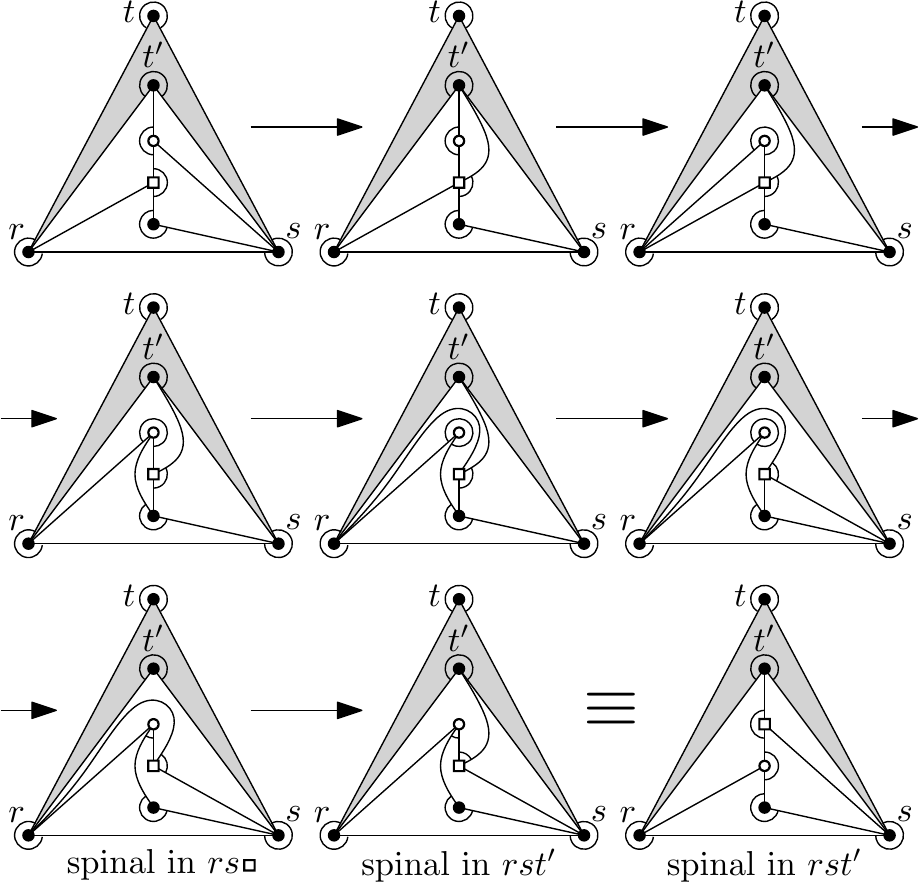}
\caption{Swapping the position of $v_2$ and $v_3$ in a combinatorial 4-PPT that is spinal in $rst'$.}
\label{fig:secondswap}
\end{figure}

\figurename~\ref{fig:generalswap} exemplifies the flip sequence for swapping the neighbored vertices $v_k$ and $v_{k+1}$ for $3\leq k\leq i-1$.
In the example $i=5$ and the labeled vertices $v_3$ and $v_4$ should be swapped, i.e., $k=3$.
For larger values of $i$ and $k$ the remaining interior vertices will be placed in the interior of $r t' s t$ (and $r s v_1 v_{k-1}$).
These areas (shown gray in \figurename~\ref{fig:firstswap},~\ref{fig:secondswap}, and~\ref{fig:generalswap}) remain untouched throughout the whole swap operation.
In this sense, all described swap operations (sequences of flips) are local.

\begin{figure}[htb]
\centering
\includegraphics{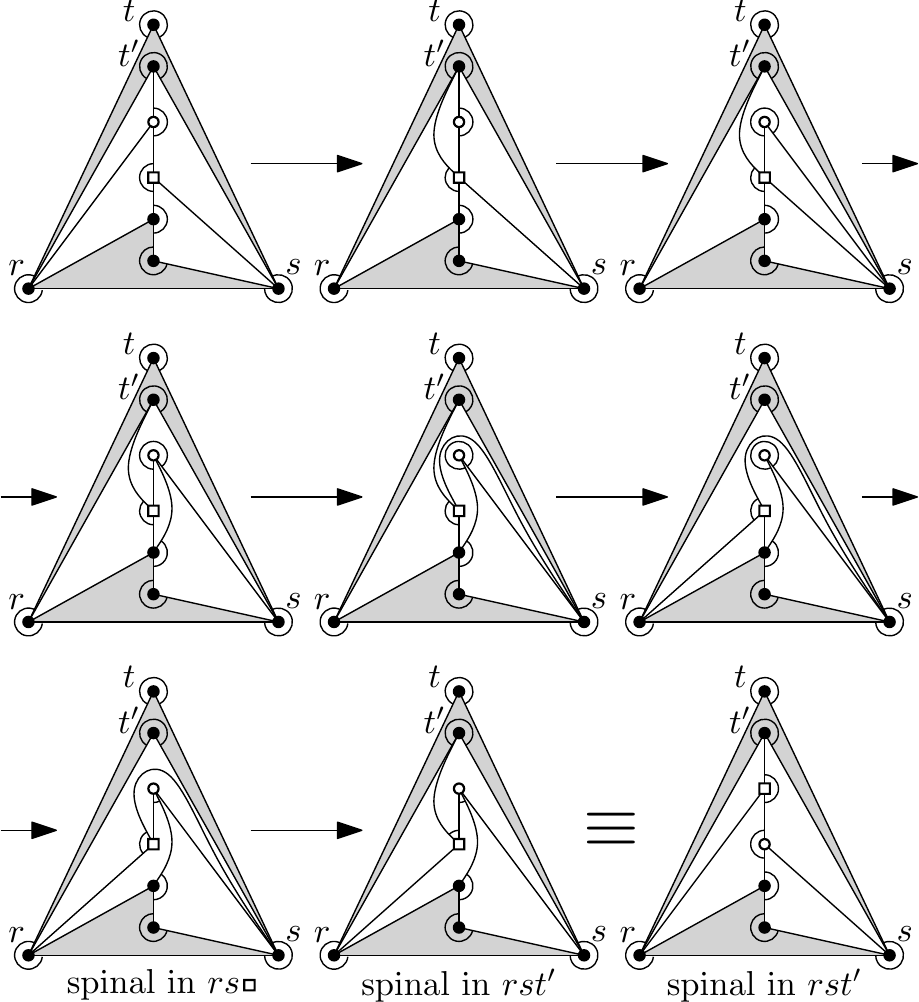}
\caption{Swapping the position of $v_k$ and $v_{k+1}$, $3\leq k\leq i-1$, in a combinatorial 4-PPT that is spinal in $rst'$.}
\label{fig:generalswap}
\end{figure}

Altogether, the presented flip sequences allow to reorder the labeled interior vertices in a canonical combinatorial 4-PPT with $O(n^2)$ flips.

\begin{theorem}
\label{thm:labeledconnectedness}
The flip graph of labeled combinatorial 4-PPTs with $n$ vertices and fixed triangular outer face is connected with diameter $O(n^2)$.
\end{theorem}

\begin{proof}
Let $T_1$ and $T_2$ be any two combinatorial 4-PPTs with $n$ vertices and triangular outer face.
By Theorem~\ref{thm:connectedness}, flipping both $T_1$ and $T_2$ to a canonical combinatorial 4-PPT $T_1'$ and $T_2'$,
respectively, takes $O(n^2)$ flips.
Except for the order of the labeled interior vertices, $T_1'$ and $T_2'$ are identical.
Thus, to flip from $T_1$ to $T_2$ we need to reorder the labeled interior vertices of $T_1'$ to match their order in $T_2'$ and then reverse the flip sequence from $T_2$ to $T_2'$.
This reordering step can also be done in $O(n^2)$ flips, by using a sorting algorithm based on exchanging pairs of neighbors.

% We use a ``bubble sort'' type algorithm:
% \begin{center}{\includegraphics{./figures/bubblesort}}\end{center}

% The inner for-loop in this algorithm starts from a spinal combinatorial 4-PPT comparing the labels of the two topmost interior vertices, $v_i$ and $v_{i-1}$.
% If they need to be exchanged, we use the swap operation to get a combinatorial 4-PPT that is spinal in $rsv_i$ with $v_i$ and $v_{i-1}$ reordered.
% Otherwise, we just apply one flip towards a canonical 4-PPT to get a combinatorial 4-PPT that is spinal in $rsv_i$ with the order of $v_i$ and $v_{i-1}$ maintained.
% Then we continue comparing the labels of the next neighbored vertices, swapping them if needed and the same time making one step towards the canonical 4-PPT.
% In other words, in the $k$-th step, $k$ from $i-1$ down to $1$, we compare the labels of $v_{k+1}$ and $v_k$, swap them if needed, and continue on a combinatorial 4-PPT that is spinal in $rsv_{k+1}$.

% Each used flip sequence needs only $O(1)$ many flips (cf.\ \figurename~\ref{fig:firstswap},~\ref{fig:secondswap}, and~\ref{fig:generalswap}).
% The loop makes $O(n)$ iterations.
% Right before each call of this inner for-loop we apply $O(n)$ flips to get a spinal combinatorial 4-PPT (Observation~\ref{obs:CanonicalSpinal}).
% This makes $O(n)$ flips in total inside the outer for-loop, which in turn again makes $O(n)$ iterations.
% Hence, after $O(n^2)$ flips the result of the algorithm is a
% combinatorial 4-PPT one flip away from $T_2'$.

%
% alternative description
%

We use a ``bubble sort'' type algorithm:
We flip the canonical combinatorial 4-PPT to the spinal combinatorial 4-PPT; i.e., a combinatorial 4-PPT that is spinal in $r s t$.
By Observation~\ref{obs:CanonicalSpinal} this transformation takes $O(n)$ flips.
We use the swap operation to exchange the two top inner vertices $v_{i}$ and $v_{i-1}$ if needed.
Applying the swap operation results in a combinatorial 4-PPT that is spinal in $r s v_{i}$.
If we do not need to exchange $v_{i}$ and $v_{i-1}$, we need one flip to get a combinatorial 4-PPT that is spinal in $r s v_{i}$.
As in bubble sort, we continue to compare and possibly exchange the next pair of neighbored vertices until $v_2$ and $v_1$ have been processed.
After every step $k$, $k$ from $i-1$ down to $1$, the two neighbored vertices $v_{k+1}$ and $v_k$ have been exchanged if needed and the resulting combinatorial 4-PPT is spinal in $r s v_{k+1}$.

After one such pass (from $k=i-1$ to $k=1$) the vertex $v_1$ is at its final position (according to its label).
Moreover, the combinatorial 4-PPT is just one flip away from the canonical one.
Further, as in one pass we move from top to bottom on the spine, each swap operation needs only $O(1)$ flips (cf.\ \figurename~\ref{fig:firstswap},~\ref{fig:secondswap}, and~\ref{fig:generalswap}).
Hence, one pass needs $O(n)$ flips.
It is easy to see that after $i-1=O(n)$ such passes every labeled vertex has been moved to its required position.
Therefore, with $O(n^2)$ flips we can reorder the labeled vertices in $T_1'$ to match their order in $T_2'$.
\end{proof}

\section{Connectivity for Combinatorial 4-PPTs with Outer Face of Arbitrary Size
   and Labeled Vertices}
\label{sec:generalconnected}

So far we have proved connectivity of combinatorial 4-PPTs with the outer face restricted to be of size three.
In this section we drop this restriction and allow outer faces of arbitrary size.
We prove that for this general case the graph of combinatorial 4-PPTs stays connected, even for labeled vertices.
The case with a triangular outer face will be a key ingredient for this proof.
To this end we define a general canonical combinatorial 4-PPT and show how to reach it with $O(n^2)$ flips.
Recall that, for labeled combinatorial 4-PPTs, we insist on the triangular outer face to be fixed throughout the flip sequence.
For larger outer faces, we also maintain this property, i.e., the vertices of the outer face are fixed and have a fixed cyclic order along the boundary in every labeled combinatorial 4-PPT along the flip sequence.
In particular, this means that the source and target labeled combinatorial 4-PPTs have to have the same sequence of labels on the boundary.

\begin{figure}[htb]
\centering
\includegraphics{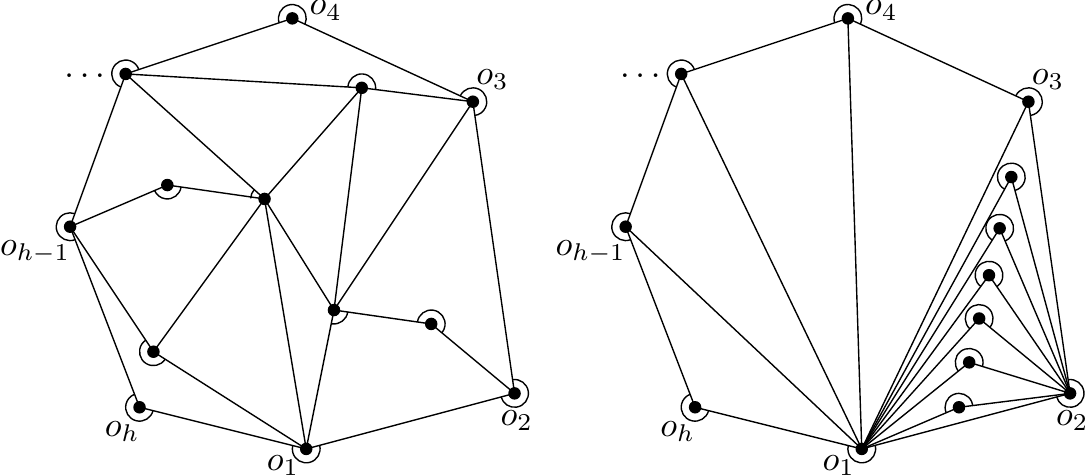}
\caption{Left: An example combinatorial 4-PPT. Right: The corresponding general canonical combinatorial 4-PPT.}
\label{fig:generalcanonical}
\end{figure}

Let $T$ be a combinatorial 4-PPT with $h$~vertices on the outer face and $i$~interior vertices ($n=h+i$).
Let the $h$~vertices on the outer face be $o_1,\ldots,o_h$ in counter-clockwise order.
Then the general canonical combinatorial 4-PPT (for $T$) consists of a triangulation on $o_1,\ldots,o_h$ with diagonals $o_1o_k$, $3\leq k\leq h-1$, (a so-called \emph{fan} at $o_1$) and a canonical combinatorial 4-PPT on all $i$ interior vertices with triangular outer face $o_1o_2o_3$ and $o_1o_2$ as the base edge.
See \figurename~\ref{fig:generalcanonical} for an example.

The flip sequence to obtain the general canonical combinatorial 4-PPT consists of three steps.
Step~1: flipping to a combinatorial 4-PPT inducing the fan at $o_1$;
Step~2: flipping to a canonical combinatorial 4-PPT inside each 3-cycle of that fan; and
Step~3: moving all interior vertices into the 3-cycle $o_1o_2o_3$ such that $o_1o_2o_3$ is the ``outer face'' of a canonical combinatorial 4-PPT with base edge $o_1o_2$.
In the following we will present each step in detail and we will show that overall $O(n^2)$ flips are sufficient.

\subsection{Step~1}

In a nutshell, we want to introduce one diagonal of the outer face after another, each time ``cutting an ear''.
In more detail, we first introduce the diagonal $o_1o_{h-1}$ of the outer face $o_1,\ldots,o_h$ to cut off the ear $o_1o_{h-1}o_{h}$.
After that we will be ready to forget about the 3-cycle $o_1o_{h-1}o_{h}$ for the moment, no matter whether its interior is empty of vertices or not.
Then we recurse on the smaller outer face $o_1,\ldots,o_{h-1}$.
This way we get a combinatorial 4-PPT containing the diagonals $o_1o_{k}$, $3\leq k\leq h-1$, in its edge set.

It remains to show how to cut an ear in a combinatorial 4-PPT, say $T$, if the diagonal $o_1o_{h-1}$ does not already exist in $T$.
This is very similar to the approach in Section~\ref{sec:GraphOfFlips}.
By Lemma~\ref{lemma:triangletochgeneral} it is always possible to move a triangular face to an arbitrary edge of the outer face.
%
% Using this result we can ensure that a triangular face $\triangle$ is incident to the edge~$o_1o_{h}$.
Thus we can ensure that a triangular face $\triangle$ is incident to the edge~$o_1o_{h}$.
If $\triangle$ is also incident to the edge~$o_{h}o_{h-1}$, then we can cut off the ear $o_1o_{h-1}o_{h}$ and iterate on the combinatorial 4-PPT with outer face $o_1,\ldots,o_{h-1}$ (note that this part contains $h-3$ triangular faces).
Otherwise, we flip away all edges incident to $o_{h}$ between $o_1o_{h}$ and $o_{h}o_{h-1}$ (inside the area not containing the reflex angle) until we can introduce the diagonal $o_1o_{h-1}$.
We explain this process in the next lemma, which is similar to Lemma~\ref{lemma:cleartip} but differs at the end of the sequence.

\begin{lemma}
\label{zzzlemma:cleartip}
Let~$T$ be a combinatorial 4-PPT with outer face $o_1,\ldots,o_h$, $h\geq4$, in which an interior triangular face~$\triangle$ is incident to the edge~$o_1o_{h}$ and the diagonal $o_1o_{h-1}$ is not an edge of $T$.
There exists a sequence of $O(n)$ flips resulting in a combinatorial 4-PPT with $o_1o_{h-1}o_{h}$ as a triangular face.
\end{lemma}
\begin{proof}
%If $o_1o_{h-1}o_{h}$ is already a triangular face of $T$, then no flip is needed.
Let $e_0,\ldots,e_{k+1}$ be the $k+2$ edges of $T$ incident to $o_{h}$ in the order of incidence, such that $e_0=o_{1}o_{h}$ and $e_{k+1}=o_{h}o_{h-1}$.
Let $\triangle = F_0\to F_1\to F_2\to\cdots\to F_k$ be the path of faces (incident to $o_{h}$) such that $e_j$, $1\leq j\leq k$, is an edge shared by the faces $F_{j-1}$ and $F_{j}$.
Note that $F_k$ is incident to $o_{h}o_{h-1}$.
With $w_j$ we denote the end vertex of $e_j$ that is not $o_{h}$.

We define a sequence of $k$ flips.
With the $j$-th flip we want to introduce a valid edge~$e'$, such that a triangular face is incident to both~$e_{j+1}$ and $o_1$.
However, this is not always possible as the insertion of the edge~$e'$ might create a double edge.
In this case, the triangular face is incident to $e_{j+1}$ and a vertex named $o_1'$, which will be used instead of $o_1$ for the remainder of the sequence of $k$ flips.
(See Cases~1 and 2 below.)
Fortunately, this can happen only once and we will carefully distinguish the different cases.
At the $j$-th flip, let $\triangle_j$ be the interior triangular face incident to $o_h$ and $o_1$ (or $o_1'$).
Consider the region $\triangle_j\cup F_j$.

\begin{figure}[htb]
\centering
\includegraphics{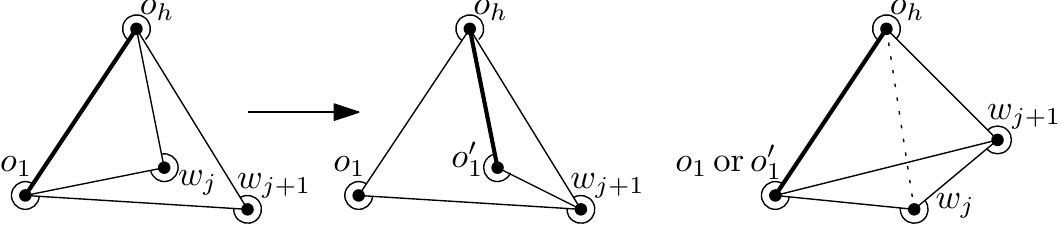}
\caption{Left (before the flip) and middle (after the flip) show the
  flip in the degenerate case. Right shows the simple case of a flip
  between two triangular faces. In bold, the edge with endpoints $o_1$ ($o_1'$ if needed) and $o_h$.}
\label{fig:degenerateearcut}
\end{figure}

\noindent \textbf{Case~1: $\bm{\triangle_j\cup F_j}$ is a degenerate 5-face.}
A flip resulting in a new triangular face that is incident to both~$e_{j+1}$ and $o_1o_{h}$ is clearly not possible.
In this case, the vertex~$w_{j}$ is renamed $o_1'$.
For the edge $o_1o_1'$ there exists a unique valid flip, by Lemma~\ref{lemma:existanceofflips}.
This results in a new triangular face that is incident to $e_{j+1}$ and $o_1'o_{h}$.
%Until the end of the flip sequence, the edge $o_1'o_{h}$ will play the role of $o_1o_{h}$.
Observe that $o_1'$ is \emph{safe}, in the sense that for each future flip, resulting in an edge~$e'$ incident to $o_1'$, $e'$ will not create a double edge. Thus, this case can only occur once.
See \figurename~\ref{fig:degenerateearcut} (left and middle).

\begin{figure}[htb]
\centering
\includegraphics{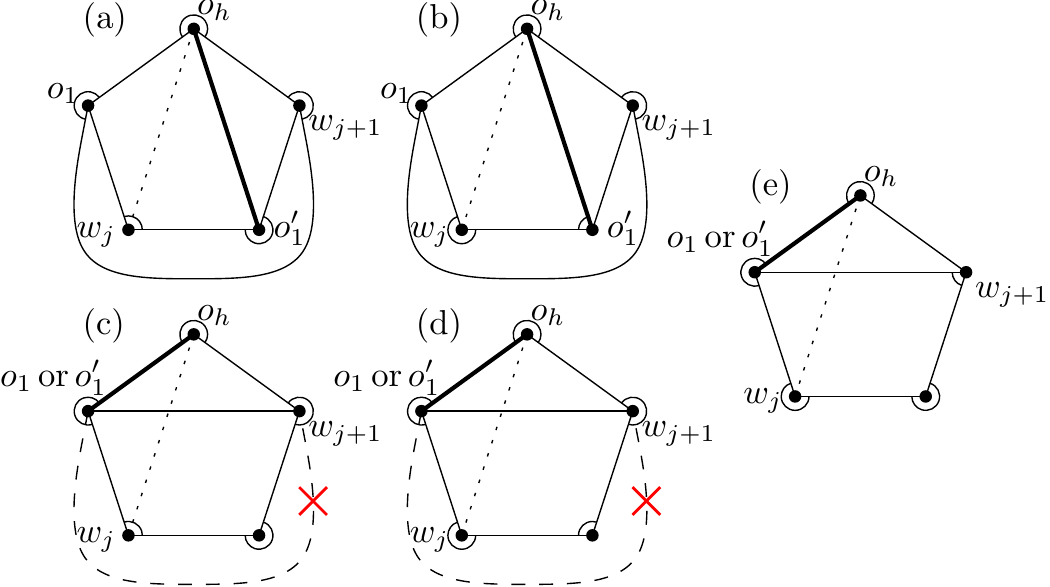}
\caption{The five subcases of the non-degenerate case, depending on the relative position of the reflex interior angle. In bold, the edge with endpoints $o_1$ ($o_1'$ if needed) and $o_h$.
  In (a) and (b)  edge $o_1w_{j+1}$ is assumed to be an edge of the combinatorial 4-PPT before the flip.
  The cases (c) and (d) are the same as (a) and (b), respectively, except that it is assumed that $o_1w_{j+1}$ is not an edge of the combinatorial 4-PPT before the flip.
  The case (e) is a flip inserting an edge incident to the reflex interior angle, which is always valid.}
\label{fig:nondegenerateearcut}
\end{figure}

\noindent \textbf{Case~2: $\bm{\triangle_j\cup F_j}$ is a non-degenerate 5-face and $\bm{o_1w_{j+1}}$ is already an edge of the combinatorial 4-PPT.}
See \figurename~\ref{fig:nondegenerateearcut}~(a) and~(b) for the two subcases, which differ in the two possible positions of the reflex angle inside $\triangle_j\cup F_j$.
In both subcases we denote as $o_1'$ the vertex of $\triangle_j\cup F_j$ that is neither $o_1$, $o_{h}$, $w_{j}$, nor~$w_{j+1}$.
Exchanging the edge $e_{j}$ with the edge $o_1'o_{h}$ is a valid flip.
(See the proof of Lemma~\ref{lemma:existanceofflips} for the non-degenerate case and use the fact that $o_1w_{j+1}$ is assumed to already be an edge of the combinatorial 4-PPT.)
Note that $o_1'$ is also safe in this case.
Therefore, Case~2 as well as Case~1 cannot occur as soon as the edge $o_1'o_{h}$ has been introduced.
That is, out of these two cases only one can occur at all and only at most once in total during the whole flip sequence.

\noindent \textbf{Case~3: $\bm{\triangle_j\cup F_j}$ is a non-degenerate 5-face and $\bm{o_1w_{j+1}}$ (or $\bm{o_1'w_{j+1}}$) is not an edge of the combinatorial 4-PPT.}
See \figurename~\ref{fig:nondegenerateearcut}~(c), (d), and~(e) for the three subcases, which differ in the position of the reflex angle inside $\triangle_j\cup F_j$.
In all three subcases there exists a valid flip introducing $o_1w_{j+1}$, such that the new triangular face is incident to both~$e_{j+1}$ and $o_1o_{h}$ (or $o_1'o_{h}$).

\noindent \textbf{Case~4: $\bm{\triangle_j\cup F_j}$ is a quadrangular face.}
Hence, $F_j$ is a triangular face.
In this case, it is easy to see that there exists a valid flip introducing $o_1w_{j+1}$, such that one of the new triangular faces is incident to both~$e_{j+1}$ and $o_1o_{h}$ (or $o_1'o_{h}$).
See \figurename~\ref{fig:degenerateearcut}~(right).

\medskip
After $k$ flips the sequence ends with a triangular face~$\triangle'$ incident to $o_ho_{h-1}$ and either $o_1o_{h}$ or $o_1'o_{h}$.
In the former case the resulting combinatorial 4-PPT has $o_1o_{h-1}o_{h}$ as a triangular face, as required.
In the latter case we need one more flip.
The edge $o_1'o_{h}$ separates $\triangle'$ from a face that is incident to $o_1o_{h}$.
We assumed the diagonal $o_1o_{h-1}$ not to be an edge of the combinatorial 4-PPT and $o_1o_{h-1}$ was also not introduced during the sequence of $k$ flips.
Therefore, replacing $o_1'o_{h}$ by $o_1o_{h-1}$ is a valid flip. See \figurename~\ref{fig:onemoreflip}.

\begin{figure}[htb]
\centering
\includegraphics{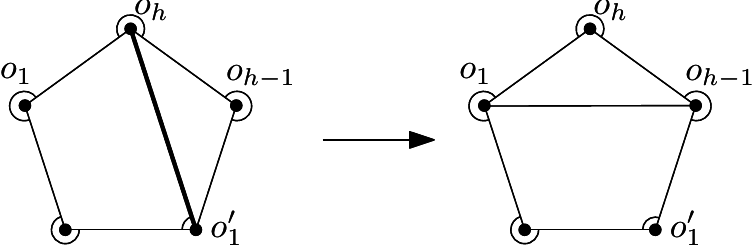}
\caption{Flip for the case needing one more flip because $o_1'o_{h}$ (bold as in previous figures) separates $\triangle'$ from a face incident to $o_1o_{h}$.}
\label{fig:onemoreflip}
\end{figure}

All in all, the sequence consists of $k$ flips, one flip per edge $e_j$, plus possibly one additional flip if the edge $o_1'o_{h}$ has been introduced during the flip sequence.
Hence, $O(n)$ flips are sufficient.
\end{proof}

Using Lemma~\ref{zzzlemma:cleartip}, we can flip to a combinatorial 4-PPT that contains the diagonals of the fan at $o_1$, by iteratively introducing the diagonals $o_1o_j$, from $j=h-1$ down to $j=3$, whenever this diagonal is not already present.

\begin{lemma}
\label{lemma:step1}
Given a combinatorial 4-PPT with outer face $o_1,\ldots,o_h$, $h\geq4$, there exists a sequence of $O(n^2)$ flips resulting in~$T$, a combinatorial 4-PPT with outer face $o_1,\ldots,o_h$, such that
the diagonals (of the outer face) $o_1o_{j}$, $3\leq j\leq h-1$, are in the set of edges of~$T$.
\end{lemma}

\subsection{Step~2}

Let $\triangle_j$ be the 3-cycle $o_1o_jo_{j+1}$, $2\leq j\leq h-1$.
After Step~1 the edges of these 3-cycles are edges of the combinatorial 4-PPT.
Let $i_j$ be the number of interior vertices inside $\triangle_j$.
By Theorem~\ref{thm:connectedness}, $O({i_j}^2)$ flips are sufficient to flip to the canonical combinatorial 4-PPT with outer face $\triangle_j$.
Therefore, overall $O(n^2)$ flips are sufficient to flip to the canonical combinatorial 4-PPT inside each 3-cycle of the fan.

\begin{lemma}
\label{lemma:step2}
Let $T$ be a combinatorial 4-PPT with outer face $o_1,\ldots,o_h$, $h\geq4$, such that
the diagonals (of the outer face) $o_1o_{j}$, $3\leq j\leq h-1$, are in the set of edges of $T$.
There exists a sequence of $O(n^2)$ flips resulting in a
combinatorial 4-PPT, $T'$, with outer face $o_1,\ldots,o_h$, such
that 1) the diagonals (of the outer face) $o_1o_{j}$, $3\leq j\leq
h-1$, are in the set of edges of $T'$, and 2) the subgraph of $T'$
inside any three-cycle $o_1o_{j}o_{j+1}$, $2\leq j\leq h-1$, is a
canonical combinatorial 4-PPT.
\end{lemma}

\subsection{Step~3}

After Step~2 the combinatorial 4-PPT with outer face
$o_1,\ldots,o_h$, $h\geq 4$, contains the diagonals $o_1o_j$,
$3\leq j\leq h-1$, in its edge set and the subgraph inside
$o_1o_jo_{j+1}$, $2\leq j\leq h-1$, is a canonical combinatorial
4-PPT with outer face $o_1o_jo_{j+1}$.
So it remains to move all interior vertices into $o_1o_2o_3$.

First consider two induced neighbored three-cycles $C_j=o_1o_jo_{j+1}$ and $C_{j+1}=o_1o_{j+1}o_{j+2}$, $2\leq j\leq h-2$.
Let $e=o_1o_{j+1}$ be the diagonal that separates $C_{j}$ and $C_{j+1}$.
Assume that the canonical combinatorial 4-PPTs in $C_{j}$ and $C_{j+1}$ have both $e$ as the base edge.

We want to move all interior vertices of $C_{j+1}$ into $C_{j}$.
It is easy to see that we can flip $e$ (as $e$ separates two
triangular faces). This results in a combinatorial 4-PPT in $C_{j}
\cup C_{j+1}$, exemplified in
\figurename~\ref{fig:movecanonical}~(a). To move one vertex from
$C_{j+1}$ to $C_{j}$ there exists a very  simple sequence of two
flips; see {\figurename}~\ref{fig:movecanonical}~(b) and~(c).
Repeatedly applying this sequence, until all vertices from
$C_{j+1}$ are moved, results in a combinatorial 4-PPT like the one
exemplified in \figurename~\ref{fig:movecanonical}~(e). We apply
one more flip to reintroduce the diagonal $e$, and all vertices
interior to $C_{j+1}$ have been moved to~$C_{j}$.

\begin{figure}[htb]
\centering
\includegraphics{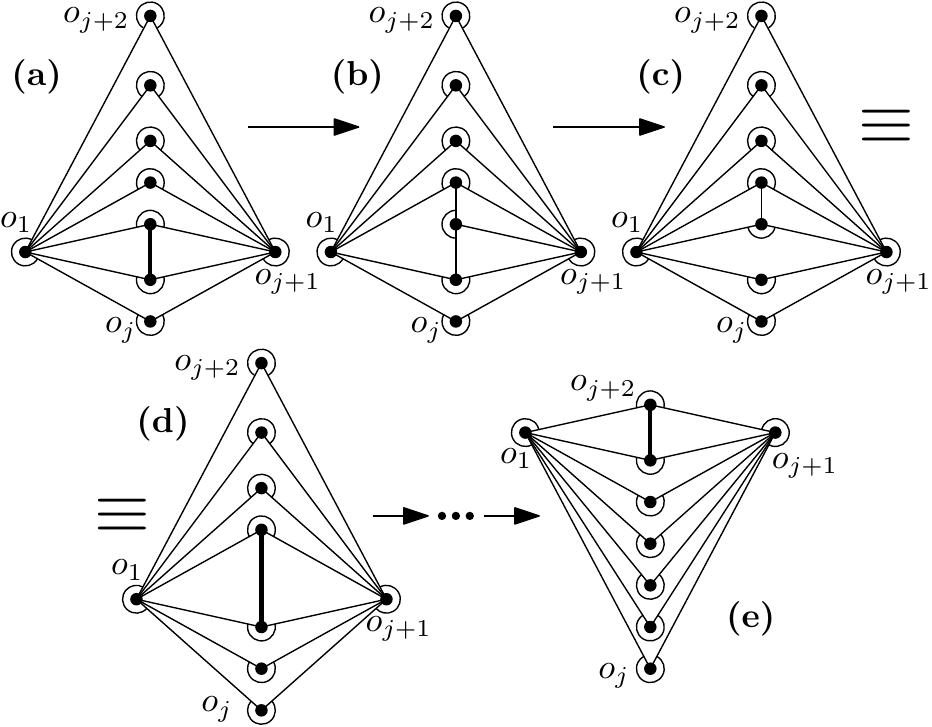}
\caption{Moving the interior vertices between two neighbored
three-cycles. (c) and (d) show different drawings of the same
graph. Comparing (a) and (d), one vertex has been moved ``down''.}
\label{fig:movecanonical}
\end{figure}

Moving all interior vertices from $C_{j+1}$ to $C_{j}$ takes $O(i_{j+1})$ flips, with $i_{j+1}$ being the number of vertices interior to $C_{j+1}$.
For moving interior vertices between neighbored three-cycles we assumed that the canonical combinatorial 4-PPTs in~$C_{j}$ and~$C_{j+1}$ both have the same base edge.
To fulfill this precondition we need a sequence of flips to rotate a canonical combinatorial 4-PPT, with triangular outer face and $i$~interior vertices, in a linear number of flips.

\begin{figure}[htb]
\centering
\includegraphics{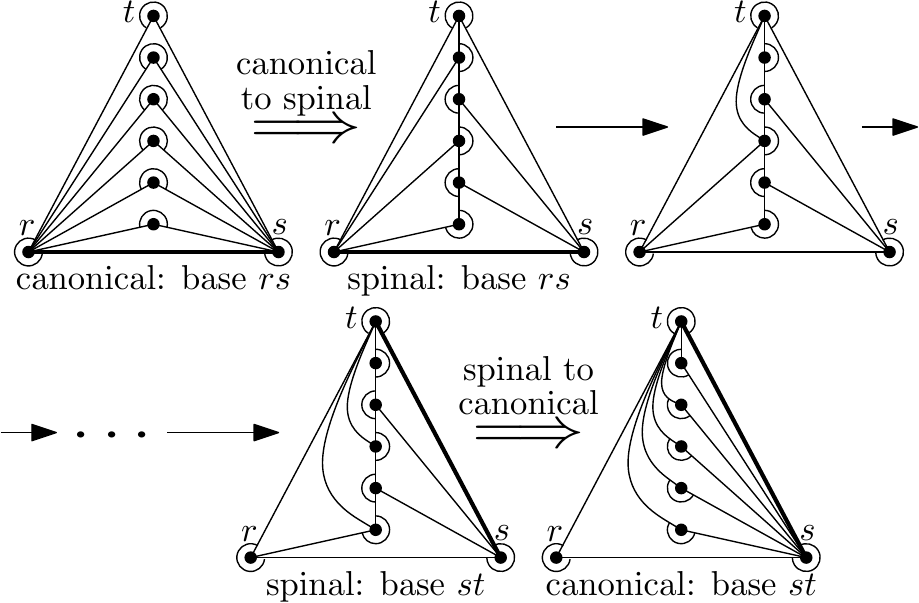}
\caption{Flip sequence for ``rotating'' a canonical combinatorial 4-PPT with triangular outer face and $i$~interior vertices.
The bold segments indicate the base edge of the canonical and spinal 4-PPT, respectively.}
\label{fig:rotatecanonicalodd}
\end{figure}

This sequence consists of:
(1) $i$~flips to obtain the spinal combinatorial 4-PPT (by Observation~\ref{obs:CanonicalSpinal});
then (2) ``rotating'' the spine with $\left\lfloor\frac{i-1}{2}\right\rfloor$~flips by flipping every other non-spinal interior edge, as depicted in \figurename~\ref{fig:rotatecanonicalodd}~(flips indicated with single arrows);
and (3) $i$~flips from spinal back to the canonical combinatorial 4-PPT with new base edge.
In \figurename~\ref{fig:rotatecanonicalodd} the whole sequence of flips is exemplified for $i$ being odd.

\begin{figure}[htb]
\centering
\includegraphics{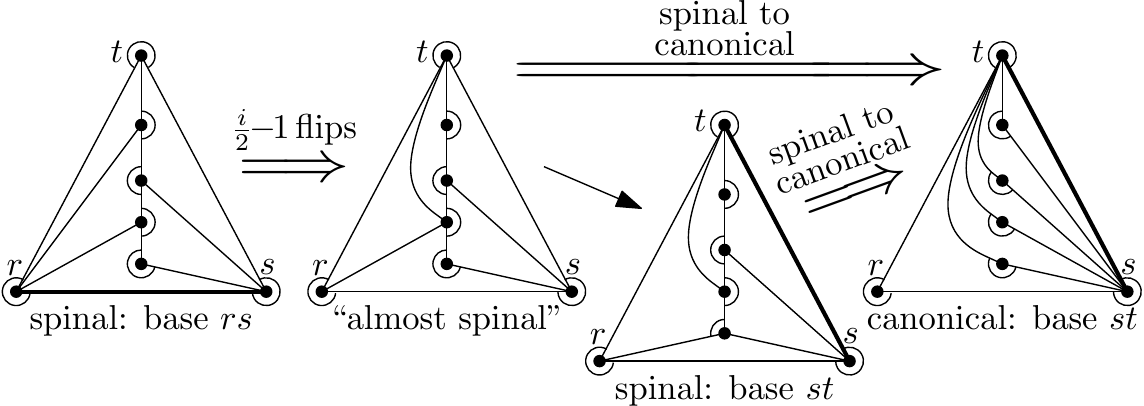}
\caption{Flip sequence for ``rotating the spine'' for an even number of $i$ interior vertices.
The bold segments indicate the base edge of the canonical and spinal 4-PPT, respectively.
The sequence of the first $\frac{i}{2}\!-\!1$ flips is the same as for $i$ odd.
The last flip to the rotated spinal 4-PPT can be eliminated, as the flip sequence from spinal to canonical can also be started from the graph named ``almost spinal''.}
\label{fig:rotatecanonicaleven}
\end{figure}

For $i$ being even, the flip sequence for ``rotating the spine'' is exemplified in \figurename~\ref{fig:rotatecanonicaleven}.
The first part of the sequence equals that for $i$ being odd.
Note that we can eliminate the last flip to the rotated spinal combinatorial 4-PPT.
The edge introduced by this last flip is the first edge that is removed in the flip sequence to the rotated canonical combinatorial 4-PPT.

Observe that we need to start with different spinal combinatorial 4-PPTs, depending on the new base edge and the parity of $i$.
If we want to rotate the canonical combinatorial 4-PPT with outer face $rst$ and base edge $rs$ to the one with base edge~$st$, then for $i$ being odd the sequence starts with flipping to the $r$-spinal combinatorial 4-PPT, and for $i$ being even the sequence starts with flipping to the $s$-spinal combinatorial 4-PPT.
See again {\figurename}~\ref{fig:rotatecanonicaleven} and {\figurename}~\ref{fig:rotatecanonicalodd}.

Starting with $j= h-2$ and stopping at $j=2$, iteratively rotating the canonical combinatorial 4-PPTs of the two neighbored three-cycles $C_j=o_1o_jo_{j+1}$ and $C_{j+1}=o_1o_{j+1}o_{j+2}$, and moving the interior vertices from $C_{j+1}$ to $C_{j}$ results in a combinatorial 4-PPT with all interior vertices inside the three-cycle $o_1o_{2}o_{3}$.
As the number of flips for the sequences of both rotating a canonical combinatorial 4-PPT (with triangular outer face) and moving the interior vertices is linear in the number of interior vertices, the overall sequence consists of $O(n^2)$ flips.

% Flipping to the canonical combinatorial 4-PPT with outer face $o_1o_{2}o_{3}$ can be done in another $O(n^2)$ flips, by Theorem~\ref{thm:connectedness}.
% This finally results in a (general) canonical combinatorial 4-PPT with outer face $o_1,\ldots,o_h$, $h\geq4$.
Finally, rotating the canonical combinatorial 4-PPT in $C_2=o_1o_{2}o_{3}$ to the base edge $o_1o_{2}$ can be done in another $O(n)$ flips.
This results in a (general) canonical combinatorial 4-PPT with outer face $o_1,\ldots,o_h$, $h\geq4$.

\begin{lemma}
\label{lemma:step3}
Let $T$ be a combinatorial 4-PPT with outer face $o_1,\ldots,o_h$,
$h\geq4$, such that
(1) the diagonals (of the outer face) $o_1o_{j}$, $3\leq j\leq h-1$, are in the set of edges of $T$,
and (2) the subgraph of $T$ inside a three-cycle $o_1o_{j}o_{j+1}$, $2\leq j\leq h-1$, is a canonical combinatorial 4-PPT.
There exists a sequence of $O(n^2)$ flips resulting in a (general) canonical combinatorial 4-PPT with outer face $o_1,\ldots,o_h$.
\end{lemma}

\subsection{General Connectivity}

Summarizing over the presented three steps and Section~\ref{sec:labeledconnectedness} about labeled vertices, we can prove the following theorem.

\begin{theorem}
\label{thm:generalconnectedness}
The graph of flips in combinatorial 4-PPTs with $n$ vertices, $h\geq3$ of them on the outer face, is connected with diameter $O(n^2)$.
This is still true for labeled combinatorial 4-PPTs with a fixed labeling on the outer face.
\end{theorem}
\begin{proof}
Let $T_1$ and $T_2$ be two combinatorial 4-PPTs with $n$ vertices, $h\geq3$ of them on the outer face.
Following the three steps summarized in Lemmas~\ref{lemma:step1},
\ref{lemma:step2}, and~\ref{lemma:step3} results in a sequence of
$O(n^2)$ flips leading to the canonical combinatorial 4-PPTs
$T_1'$ and $T_2'$, respectively, with outer face $o_1,\ldots,o_h$
(see \figurename~\ref{fig:generalcanonical}).

In the unlabeled case $T_1'= T_2'$.
As all used flips are invertible this proves that the flip graph is connected with diameter $O(n^2)$.

In the case of labeled vertices, we can flip from $T_1'$ to $T_2'$ with $O(n^2)$ flips, by Theorem~\ref{thm:labeledconnectedness}.
Hence, the flip graph is connected with diameter $O(n^2)$ in the labeled case, too.
\end{proof}

\section{Lower Bounds}
For unlabeled graphs, we are not aware of any lower bound for the diameter of the flip graph other than the trivial linear one.
For labeled combinatorial 4-PPTs, we provide a reduction from the $\Omega(n \log n)$ lower bound for combinatorial triangulations.
Sleator, Tarjan, and Thurston~\cite{labeled_triang} prove the lower bound for the flip distance between two so-called \emph{double wheels}, which are isomorphic, but labeled differently. (A double wheel consists of a cycle of $n-2$ vertices, plus two vertices that are each connected to all vertices of the cycle.)
We show that a short flip sequence between two combinatorial 4-PPTs could be used to find a flip sequence between these two triangulations that is longer only by a constant factor.

For a given combinatorial 4-PPT $T$ with triangular outer face, let $I(T)$ be the graph obtained from the following operation:
Inside each 4-gon, add an edge from the reflex vertex to the opposite one.
We call $I(T)$ the \emph{induced triangulation} of $T$.

\begin{lemma}\label{lem_induced}
$I(T)$ is a combinatorial triangulation.
\end{lemma}
\begin{proof}
As each face in $I(T)$ is triangular, it remains to show that $I(T)$ is simple.
% Since there is no 4-gon in~$T$ where the reflex vertex~$v_r$ is identical with the opposite one (called $v_o$), adding edges in the described way cannot produce loops.
% It remains to show that we do not produce multiple edges between two vertices.
% Suppose first that there is a 4-gon~$F$ such that $T$ contains an edge $e$ between $v_r$ and $v_o$ (which is outside $F$).
% Then two edges of $F$ and $e$ produce a 3-cycle with a reflex angle in its interior, a contradiction to $T$ being a combinatorial 4-PPT.
% Hence suppose that in the process of adding edges we added an edge~$e$ between $v_r$ and $v_o$ outside~$F$.
% Then this would mean that there is another face~$F'$ in $T$ with $v_o$ as its reflex vertex and $v_r$ being its opposite one in~$F'$.
% But this would imply a 4-cycle with two inner reflex angles in~$T$, a contradiction.
Consider $T$ to be embedded as a (straight-line) pointed pseudo-triangulation (see Theorem~\ref{theorem:stretched}).
Inside each 4-gon, we can add the edge for $I(T)$ as a straight line segment in the described way, as the reflex vertex always ``sees'' the opposite one.
The resulting graph is geometric and therefore simple.
\end{proof}

\begin{figure}[htb]
\centering
\includegraphics{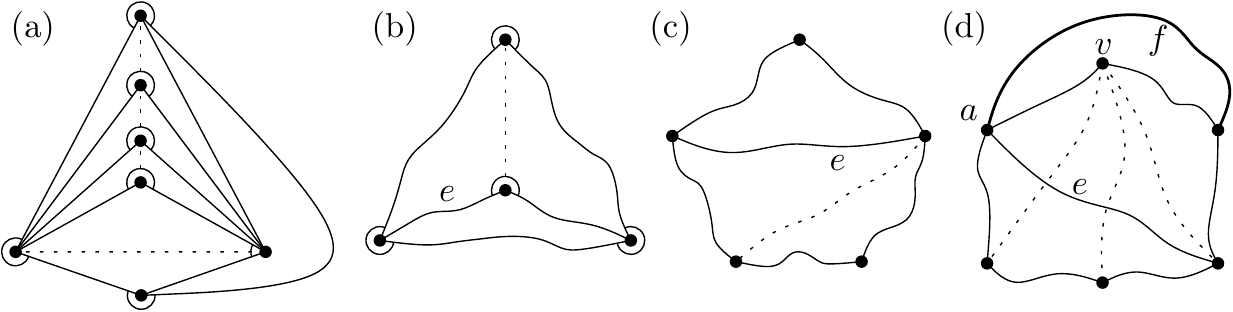}
\caption{Lower bound illustrations. (a) The source and target combinatorial 4-PPT without labels and its induced triangulation (dotted edges); when flipping the bottom-most dotted edge we obtain a double wheel. (b) No flip is necessary when the flipped edge is incident to two faces whose union is a triangle. (c) Inside a 5-gon every triangulation is a fan triangulation. (d) If not all edges can be flipped to be incident to $a$, it is because of an edge $f$ (shown bold). This edge $f$ ensures that a fan at~$v$ exists.}
\label{fig:lower_bound}
\end{figure}

\begin{theorem}
The graph of combinatorial 4-PPTs with labeled vertices has diameter $\Omega(n \log n)$.
\end{theorem}
\begin{proof}
Consider the graph shown in \figurename~\ref{fig:lower_bound}~(a).
Observe that the induced triangulation of this 4-PPT is only one flip away from the double wheel.
Hence, the flip distance between two induced labeled triangulations is asymptotically the same for the double wheel and the triangulation shown in \figurename~\ref{fig:lower_bound}~(a).

To prove the theorem, we use the fact that there exist two different labelings of the double wheel such that their flip distance is $\Omega(n \log n)$, see ~\cite{labeled_triang}.
Thus, such a pair of labelings, with flip distance $\Omega(n \log n)$, also exists for the induced triangulation shown in \figurename~\ref{fig:lower_bound}~(a).
Suppose that between every pair of labeled combinatorial 4-PPTs there exists a flip sequence of length $o(n \log n)$.
We will prove that for each single flip in such a sequence, leading from a labeled combinatorial 4-PPT $T$ to a labeled combinatorial 4-PPT $T'$, there exists a sequence of flips from $I(T)$ to $I(T')$ that has constant length.
This leads to a contradiction that proves the theorem.

Let $e$ be the edge in~$T$ that is flipped to the edge $e'$ in the resulting combinatorial 4-PPT~$T'$.
Let $F$ and $\tilde{F}$ be the faces incident to $e$ in~$T$.
Note that $I(T)$ and $I(T')$ both are combinatorial triangulations due to Lemma~\ref{lem_induced}.
Further, observe that $I(T)$ and $I(T')$ are equivalent outside $F\cup\tilde{F}$.
Therefore, for each edge $f$ of $I(T')$ inside $F\cup\tilde{F}$ (in particular also for $f=e'$) there cannot exist an edge outside $F\cup\tilde{F}$ in $I(T)$ connecting the vertices of $f$.

We now distinguish four cases depending on the size of $F\cup\tilde{F}$.
If $F\cup\tilde{F}$ is a triangle (with one interior vertex), then $I(T) = I(T')$ and we do not have to perform any triangulation flip; see \figurename~\ref{fig:lower_bound}~(b).

The case that $F\cup\tilde{F}$ is a 4-gon cannot happen: If $F$ and $\tilde{F}$ were two 4-faces sharing two edges (among them $e$), then $e$ would not be flippable in $T$, as one of the resulting faces of the flip would have five vertices. The second possibility would be that $F$ and $\tilde{F}$ are two triangles, but as $T$ has only three vertices on the outer face, it has only one interior face that is a triangle.
%If $F\cup\tilde{F}$ is a 4-gon, then $e$ cannot be incident to two 4-gons such that $F$ and $\tilde{F}$ share two edges, because then $e$ would not be flippable in $T$, as this would result in one face being a 5-gon.
%Thus, $e$ is incident to two triangles in this case and flipping from~$I(T)$ to~$I(T')$ is the very same as flipping from~$T$ to~$T'$.

Next, suppose that $F\cup\tilde{F}$ is a 5-gon, i.e., $F$ and  $\tilde{F}$ are a triangle and a 4-face, respectively, as shown in \figurename~\ref{fig:lower_bound}~(c).
Note that every triangulation of a 5-gon is a fan triangulation (i.e., all diagonals are incident to the same vertex).
It is well-known (see, e.g., ~\cite{sleator}) that between two triangulations, of which at least one is a fan triangulation, there exists a flip sequence such that each flip results in an edge of the target triangulation.
Hence, there exists a flip sequence of at most two flips from~$I(T)$ to~$I(T')$.

The last case is that $F\cup\tilde{F}$ is a 6-gon having $e$ and $e'$ as (crossing) diagonals, with $F$ and $\tilde{F}$ being two 4-faces.
Our goal is to obtain a fan triangulation inside $F\cup\tilde{F}$, which allows us to flip in the missing edges of~$I(T')$ one by one (recall the arguments for the 5-gon case).
We try to flip to a triangulation in which all edges inside $F\cup\tilde{F}$ are incident to some vertex~$a$ of~$e$.
If this is not possible, then there exists an edge~$f$ outside $F\cup\tilde{F}$ that prevents one edge of the fan at~$a$; see \figurename~\ref{fig:lower_bound}~(d).
In the exterior of $F \cup \tilde{F}$, $f$ separates one vertex~$v$ of the 6-gon $F\cup\tilde{F}$ from the remaining three.
We can therefore flip to a triangulation with all edges incident to~$v$ (resulting in a fan triangulation inside $F \cup \tilde{F}$).
Hence, there exists a flip sequence of length at most six between~$I(T)$ and~$I(T')$.
\end{proof}

The close resemblance of the problem to the one of combinatorial triangulations suggests that the same bounds hold for combinatorial 4-PPTs.
We expect the upper bounds to be improvable if more insight on the structure of combinatorial 4-PPTs is obtained.

\section*{Acknowledgments}

This work was initiated during the 9\textsuperscript{th} European Research Week on Geometric Graphs and Pseudo-Triangulations, held May 14--18, 2012 in Alcal\'a de Henares, Spain.
We thank Vincent Pilaud, Pedro Ramos, and Andr\'e Schulz for helpful comments.

Oswin Aichholzer and Birgit~Vogtenhuber are partially supported by the ESF EUROCORES programme EuroGIGA -- CRP `ComPoSe', Austrian Science Fund (FWF): I648-N18.
Thomas Hackl is supported by the Austrian Science Fund (FWF): P23629-N18 `Combinatorial Problems on Geometric Graphs'.
David Orden is partially supported by MICINN Project MTM2011-22792, ESF EUROCORES programme EuroGIGA - ComPoSe IP04 - MICINN Project EUI-EURC-2011-4306, Comunidad de Madrid Project S2013/ICE-2919 TIGRE5-CM, and Junta de Castilla y Le\'on Project VA172A12-2 ATLAS.
Alexander Pilz is recipient of a DOC-fellowship of the Austrian Academy of Sciences at the Institute for Software Technology, Graz University of Technology, Austria.
Maria Saumell is supported by the project NEXLIZ – CZ.1.07/2.3.00/30.0038, which is co-financed by the European Social Fund and the state budget of the Czech Republic, by ESF EuroGIGA project ComPoSe as F.R.S.-FNRS - EUROGIGA NR 13604, and by ESF EuroGIGA project GraDR as GA\v{C}R GIG/11/E023.

\bibliographystyle{abbrv}
\bibliography{bibliography}

\end{document}